\documentclass[a4paper,11pt]{article}

\usepackage[utf8]{inputenc}
\usepackage{amssymb,amsmath,mathrsfs,amsthm,bbm,xcolor}
\usepackage{verbatim}
\usepackage{indentfirst}
\usepackage{geometry}
\usepackage{enumerate} \usepackage{graphicx} 

\usepackage{booktabs,caption}
\usepackage[flushleft]{threeparttable}
\usepackage{hyperref} 
\hypersetup{colorlinks=true,allcolors=blue}
\usepackage{hypcap}

\usepackage[capitalize]{cleveref} 

\newtheorem{thm}{Theorem}[section]
\newtheorem{prop}[thm]{Proposition}
\newtheorem{cor}[thm]{Corollary}

\newtheorem{lem}[thm]{Lemma}
\newtheorem*{lem*}{Lemma}

\newtheorem*{claim*}{Claim}

\theoremstyle{remark}
\newtheorem{rem}[thm]{Remark}
\newtheorem*{rem*}{Remark}

\theoremstyle{definition}
\newtheorem{defi}[thm]{Definition}

\makeatletter
\newtheorem*{rep@theorem}{\rep@title}
\newcommand{\newreptheorem}[2]{%
\newenvironment{rep#1}[1]{%
 \def\rep@title{#2 \ref{##1}}%
 \begin{rep@theorem}}%
 {\end{rep@theorem}}}
\makeatother
\newreptheorem{thm}{Theorem}
\newreptheorem{lem}{Lemma}

\input{macros}

\begin{document}
	
\bibliographystyle{alpha}
\title{\textbf{On the dimension of limit sets on $\P(\R^3)$ via stationary measures: variational principles and applications}}
\author{Yuxiang Jiao, Jialun Li, Wenyu Pan and Disheng Xu}
\date{}
\maketitle
\begin{abstract}

In this article, we establish the variational principle of the affinity exponent of Borel Anosov representations. We also establish such a principle of the Rauzy gasket. In \cite{LPX},  they obtain a dimension formula of the stationary measures on $\PP(\RR^3)$. Combined with our result, it allows us to study the Hausdorff dimension of limit sets of Anosov representations in $\SL_3(\R)$ and the Rauzy gasket. It yields the equality between the Hausdorff dimensions and the affinity exponents in both settings. In the appendix, we improve the numerical lower bound of the Hausdorff dimension of Rauzy gasket to $1.5$.


\end{abstract}

\tableofcontents

\section{Introduction}
This article is the second part of \cite{LPX}. We provide different examples of computing the Hausdorff dimension of limit sets on a projective space using stationary measures.  

First, we consider the limit sets of Anosov representations. 
For a finitely generated hyperbolic group $\Gamma,$ let $|\cdot|$ be the word norm on $\Gamma$ with respect to a fixed finite generating set. 
A homomorphism $\rho:\Gamma\to\SL_n(\RR)$ is called a \textit{Borel Anosov representation} if there exists $c>0$ such that for every $1\leqslant p\leqslant n-1,$
 \[\frac{\sigma_p(\rho(\gamma))}{\sigma_{p+1}(\rho(\gamma))}\geqslant ce^{c|\gamma|},\quad\forall \gamma\in\Gamma ,\]
where $\sigma_p(g)$ denotes the $p$-th maximal singular value of $g\in\SL_n(\RR).$

The concept of Anosov representation was first introduced by Labourie in \cite{labourie_anosov_2006} to study Hitchin components of the representations of surface groups. They are a generalization of convex cocompact representations in $\mr{PSO}(n,1).$
Given a Borel Anosov representation $\rho,$ we consider the action of $\rho(\Gamma)$ on the projective space $\PP(\RR^n)$. It always admits a limit set, denoted by $L(\rho(\Gamma)),$ which is the closure of attracting fixed points of $\rho(\gamma)$ on $\PP(\RR^n)$ for $\gamma\in\Gamma$.

To compute the Hausdorff dimension of the limit set of $\rho(\Gamma)$, the difficulty lies in the fact that $\rho(\Gamma)$ usually acts on $\PP(\RR^n)$ non-conformally. For example, consider the action of diagonal matrix $\mathrm{diag}(2,1,1/2) $ on $\P(\R^3)$. The stretch rates in different directions are different. When computing the Hausdorff dimension of the set, one needs to estimate the minimal number of balls to cover it. Naturally, in the non-conformal setting, an efficient way is to arrange the balls according to the stretch rate. Therefore, it leads us to consider the \textit{affinity exponent} and we expect the Hausdorff dimension of the limit set equals its affinity exponent. 
The concept of affinity exponent was proposed by Falconer to study the Hausdorff dimension of self-affine fractals \cite{falconer_hausdorff_1988}. Later, Pozzetti-Sambarino-Wienhard \cite{pozzetti_anosov_nodate} extended this concept to study Anosov representations.

For $0\leqslant s\leqslant n-1,$ let $\psi_s:\SL_n(\RR)\to\RR$ be given by
\[\psi_s(g)\defeq\sum_{1\leqslant i\leqslant \lfloor s\rfloor }(\log\sigma_1(g)-\log\sigma_{i+1}(g))+(s-\lfloor s\rfloor)(\log\sigma_1(g)-\log\sigma_{\lfloor s\rfloor+2}(g)),\quad g\in\SL_n(\RR).\]
Then the affinity exponent of $\rho$ is given by
\begin{equation}\label{eqn: critical exponent}
	s_{\mathrm{A}}(\rho)\defeq\sup\lb{s: \sum_{\gamma\in\Gamma}\exp(-\psi_{s}(\rho(\gamma)))=\infty }.
\end{equation}

It is shown in \cite{pozzetti_anosov_nodate} that the affinity exponent is always an upper bound of the Hausdorff dimension of the limit set. 
To obtain the equality between two notions of dimensions, it remains to show a reversed inequality.
An approach to give a lower bound of the dimension of $L(\rho(\Gamma))$ is to consider measures supported on it. A probability Borel measure $\mu$ on a metric space is called \textit{exact dimensional} if there exists $\alpha$ such that 
\[\lim_{r\to 0}\frac{\log\mu(B(x,r))}{\log r}=\alpha,\quad \mu\mr{-a.e.}~x ,\]
and $\alpha$ is called the exact dimension of $\mu,$ which will be denoted by $\dim\mu$.
Due to a result by Young \cite{young_dimension_1982}, the Hausdorff dimension of a set is bounded below by the exact dimension of measures (if exist) supported on it. Therefore, our approach is to construct satationary measures supported on the limit set whose exact dimensions approximate the affinity exponent.

Let us recall the definition of Lyapunov dimension of stationary measures, which is the expected value of exact dimensions. Let $\nu$ be a finitely supported probability measure on $\SL_n(\RR)$ with the Lyapunov spectrum $\lambda(\nu)=\lb{ \lambda_1(\nu)\geqslant\cdots\geqslant\lambda_n(\nu) }.$
We denote $\chi_i(\nu)=\lambda_1(\nu)-\lambda_{i+1}(\nu)$ for $1\leqslant i\leqslant n-1.$
Let $\mu$ be a $\nu$-stationary measure on $\PP(\RR^n).$
The \textit{Furstenberg entropy} is given by
\begin{equation}\label{equ:f entropy}
    h_{\mathrm{F}}(\mu,\nu)=\int \log\frac{\dd g\mu}{\dd\mu}(\xi) \left(\frac{\dd g\mu}{\dd\mu}(\xi) \right)\dd\nu(g)\dd\mu(\xi).
\end{equation}
Assume further that $\chi_{n-1}(\nu)>0,$ then the Lyapunov dimension of $\mu$ is defined to be
\begin{equation}\label{eqn: Lyapunov dimension}
	\dim_{\mr{LY}}\mu= d+\frac{h_{\mr{F}}(\mu,\nu)-(\chi_1(\nu)+\cdots+\chi_d(\nu))}{\chi_{d+1}(\nu)}, 
\end{equation}
where $d$ is the maximal integer such that $\chi_1(\nu)+\cdots+\chi_d(\nu)\leqslant h_{\mr{F}}(\mu,\nu).$ 

To achieve the affinity exponent in our setting, we will approximate it by Lyapunov dimensions of stationary measures. We call it a \textit{variational principle of critical exponent}. Such variational principle has been considered in several different contexts to obtain a lower bound of the Hausdorff dimension of dynamically invariant sets: Morris-Shmerkin \cite{morris_equality_2016} and Morris-Sert \cite{morris_variational_2023} on self-affine IFSs on $\RR^n$ and He-Jiao-Xu \cite{HJX} on $\Diff(\SS^1)$. 

In conclusion, our approach can be roughly summarized by the following inequalities
\[ s_{\mathrm{A}}(\rho)\stackrel{\text{(1)}}{\geqslant}\dim L(\rho(\Gamma))\stackrel{\text{(2)}}{\geqslant} \sup_\mu\dim\mu \stackrel{\text{(3)}}{=}\sup_\mu\dim_{\mr{LY}}\mu\stackrel{\text{(4)}}{\geqslant} s_{\mathrm{A}}(\rho);\]
here (1) is established in \cite{pozzetti_anosov_nodate}, (2) is established in \cite{young_dimension_1982}, \cite{rapaport_exact_2021} and \cite{ledrappier_exact_2021}, (3) is established in \cite{hochman_dimension_2017} and \cite{LPX} for special cases. We obtain the following variational principle, which establishes (4). We say $\rho:\Gamma\to\SL_n(\RR)$ is Zariski dense if $\rho(\Gamma)$ is Zariski dense in $\SL_n(\RR).$
\begin{thm}
\label{thm: variational principle}
    Let $\rho:\Gamma\to\SL_n(\RR)$ be a Zariski dense Borel Anosov representation. For every $\ve>0,$ there exists a finitely supported probability measure $\nu$ on $\rho(\Gamma)$ whose support generates a Zariski dense subgroup such that the unique $\nu$-stationary measure $\mu$ satisfying 
    \[\dim_{\mr{LY}}\mu\geqslant s_{\mathrm{A}}(\rho)-\ve.\]
\end{thm}

To obtain the approximations, we need to find a probability measure $\nu$ on $\rho(\Gamma)$ with large entropy and controlled Lyapunov exponents. In our setting, the Furstenberg entropy equals the random walk entropy \eqref{eqn: random walk entropy}, which characterizes the freeness of the semigroup generated by $\supp\nu.$
We aim to find a finite subset of $\rho(\Gamma)$ which freely generates a free semigroup, and consider the uniform random walk on this subset.

At the core of the proof, we use a geometric group theoretic argument to construct free semigroups, which is different from the one in the IFS setting. 
Thanks to the hyperbolicity of the group, we can always approximate the group by free semigroups, in the sense of the growth rate of groups.
However, the approximating semigroups should satisfy some additional conditions coming from the dynamics. 
Our key argument, as presented in \cref{sec: geometric group thry,sec:limitset}, establishes such a desired construction.

A direct application of Theorem \ref{thm: variational principle} is computing the Hausdorff dimension of limit sets of Anosov representations in $\SL_3(\RR).$
Let us recall the dimension formula of stationary measures established in the first part of our paper \cite{LPX}.

\begin{thm}[{\cite[Theorem 1.10]{LPX}}]\label{thm:lyapunov}
    Let $\nu$ be a Zariski dense, finitely supported probability measure on $\SL_3(\R)$ that satisfies the exponential separation condition, and $\mu$ be its Furstenberg measure on $\P(\R^3)$. Then we have $\dim\mu=\dim_{\rm{LY}}\mu. $
\end{thm}

As a consequence, we can derive the Hausdorff dimension of limit sets, as shown in \cite{LPX}.
\begin{thm}[{\cite[Theorem 1.3]{LPX}}]
	Let $\Gamma$ be a hyperbolic group and $\rho:\Gamma\rightarrow\SL_3(\R)$ be an irreducible Anosov representation, then the Hausdorff dimension of the limit set $L(\rho(\Gamma))$ in $\P(\R^3)$ equals the affinity exponent $s_{\mathrm{A}}(\rho)$.
\end{thm}

In a similar vein of Theorem \ref{thm: variational principle}, we also obtain a variational principle on the flag variety, Proposition \ref{prop: flag variation}. An application is to obtain an estimation of the Hausdorff dimension of limit sets of Borel Anosov representations on the flag variety by Ledrappier-Lessa \cite{LedLe232} .

\paragraph{Rauzy gasket}
Another example we consider is the \textit{Rauzy gasket} which is a fractal set on $\PP(\RR^3)$ formed by projective actions of $\SL_3(\RR)$.
We also establish the identity between its Hausdorff dimension and its affinity exponent: we show the affinity exponent is an upper bound of its Hausdorff dimension and a variational principle of the affinity exponent.

Let $\Delta$ be the projectivization of $\{(x,y,z): x,y,z\geq 0\}$ in $\PP(\RR^3).$
Let $\Gamma_{\mathscr{R}}$ be the semigroup generated by 
    \[A_1=\begin{pmatrix} 1 & 1 & 1 \\ 0 & 1 & 0\\ 0& 0& 1\end{pmatrix},\  A_2=\begin{pmatrix} 1 & 0 & 0 \\ 1 & 1 & 1\\ 0& 0& 1\end{pmatrix},\  A_3=\begin{pmatrix} 1 & 0 & 0 \\ 0 & 1 & 0\\ 1& 1& 1\end{pmatrix},\] 
    and we call it the Rauzy semigroup.
    Then as $\Gamma_{\mathscr{R}}\subset\SL_3(\R)$, the semigroup $\Gamma_{\mathscr{R}}$ acts on $\P(\R^3)$. Due to the choice of $\Delta$, the semigroup $\Gamma_{\mathscr{R}}$ preserves $\Delta$. The Rauzy gasket $X$ is the unique attractor of the Rauzy semigroup, which can be defined formally as
    \[\bigcap_{n\rightarrow\infty}\bigcup_{i_j\in\{1,2,3\}} (A_{i_1}\cdots A_{i_{n}}\Delta). \]

The Rauzy gasket, depicted in Figure 1, was first introduced in 1991 by Arnoux and Rauzy \cite{arnoux_geometric_1991} in the context of interval exchange transformations. They conjectured that the gasket has Lebesgue measure 0. Levitt \cite{levitt1993dynamique} rediscovered the gasket in 1993 and confirmed the Arnoux-Rauzy conjecture, the proof of which is essentially due to Yoccoz. Later, in the study of Novikov's problem, Dynnikov and De Leo \cite{deleo2009geometry} provided a numerical estimate of the Hausdorff dimension of the Rauzy gasket. They suggested lower and upper bounds are $1.7$ and $1.8$, respectively. Meanwhile, Arnoux asked whether the Hausdorff dimension is less than or equal to $2$ \cite{arnoux_rauzy_2013}. Avila, Hubert, and Skripchenko \cite{avila_hausdorff_2016} provided a positive answer to this question, and Guti\'errez-Romo and Matheus in \cite{gutierrez-romo_lower_2020} proved the lower bound is greater than $1.19$.
Recently, Pollicott and Sewell \cite{pollicott_upper_2021} used a renewal theoretical argument to show that the Hausdorff dimension of the Rauzy gasket is less than $1.7407$. See also \cite{fougeron2020dynamical} for a weaker upper bound.
\begin{figure}[!ht]
\centering
\includegraphics[width=0.43\textwidth]{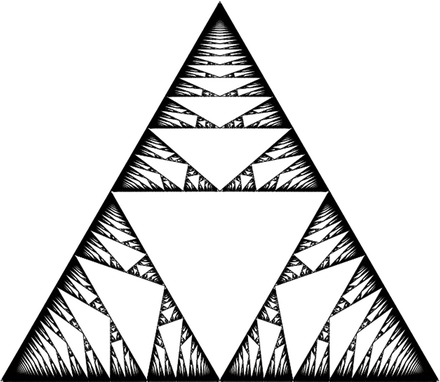}
\caption{Figure in \cite{arnoux_rauzy_2013}}
\end{figure}

Let $\rho$ be the natural embedding of Rauzy semigroup $\Gamma_{\mathscr{R}}$ into $\SL_3(\RR).$
The affinity exponent $s_{\mathrm{A}}(\Gamma_{\mathscr{R}})$ of the Rauzy gasket is defined as the critical exponent of \eqref{eqn: critical exponent}, which also works for the Rauzy semigroup.
The following theorem confirms a folk-lore conjecture of the Hausdorff dimension of the Rauzy gasket. 

\begin{thm}\label{thm:rauzy}
The Hausdorff dimension of the Rauzy gasket $X$ is equal to its affinity exponent $s_{\mathrm{A}}(\Gamma_{\mathscr{R}}).$ 
\end{thm}

The result and the proof of Theorem have an interesting outgrowth: we 
improve the numerical lower bound obtained in \cite{gutierrez-romo_lower_2020} to $3/2$, and the argument is versatile and allows us to deal with 
a (semi)group which contains a large reducible subsemigroup.

\begin{cor}\label{cor: 1.5}We have $\dim X\geq 3/2$.
\end{cor}


Unlike Anosov representations, the Rauzy semigroup is non-uniformly hyperbolic. The estimate of the upper bound of the Hausdorff dimension requires the study of the points where the action lacks hyperbolicity. 
To obtain a lower bound of the Hausdorff dimension, we also establish a variational principle for the Rauzy semigroup. We build on the variational principle for IFS setting \cite{morris_variational_2023}. Like Anosov representations, we aim to find free semigroups. In this process, the difficulty lies in that non-uniform hyperbolicity makes us lose control of word lengths. We make use of the \textit{prefix argument} to overcome the issue. This prefix argument also occurs in \cite{HJX}.

\begin{rem} Recently
 Natalia Jurga also obtained \cref{thm:rauzy} independently.
\end{rem}

\paragraph{Organization.} In Section \ref{sec:entropy}, we discuss different notions of entropy. We establish a geometric group theoretic lemma in Section \ref{sec: geometric group thry} and the variational principle for Anosov representations in Section \ref{sec:limitset}. Section \ref{sec:rauzy} is devoted to the study of the Hausdorff dimension of the Rauzy gasket. 

\paragraph{Acknowledgement.}
We would like to thank Wenyuan Yang for carefully explaining the basic ideas and arguments in \cite{Yang19}, which is useful for \cref{sec: geometric group thry}.
We would like to thank Cagri Sert, Fran\c cois Ledrappier, and Pablo Lessa for helpful discussions. Part of the work was done in the conference  ``Beyond uniform hyperbolicity'' at the Banach Center in B\k{e}dlewo, Poland, in 2023. We thank the organizers and the hospitality of the center.

\section{Preliminaries}\label{sec:entropy}

\subsection{Actions and random walks of \texorpdfstring{$\SL_n(\RR)$}{SL\_n(R)}}\label{subsec: actions}

Consider the $n$-dimensional Euclidean space $\R^n$ and denote by $\|\cdot\|$ the Euclidean norm. By abuse of the notation, we denote by $\|\cdot\|$ the norm on $\wedge^2\R^n$ induced by the one in $\R^n$. 
Let $\PP(\RR^n)$ be the projective space. The metric on $\PP(\RR^n)$ is given by
\[d(\R v,\R w):=\frac{\|v\wedge w\|}{\|v\|\|w\|}\,\,\,\text{for any}\,\,\, \R v,\R w\in \P(\R^n),\]
which is bi-Lipschitz equivalent to the $\SO(n)$-invariant metric on $\PP(\RR^n).$

Let $\frak a=\lb{\lambda=\diag(\lambda_1,\cdots,\lambda_n):\lambda_i\in\RR,\sum_i\lambda_i=0}$ be a Cartan algebra of $\fsl_n(\RR)$ and $\frak a^+=\lb{\lambda\in\frak a:\lambda_1\geqslant\cdots\geqslant\lambda_n}$ be a positive Weyl chamber. Set $A^+=\exp\frak a^+$ and $K=\SO_n(\RR).$

For every $g\in \SL_n(\RR),$ it admits the Cartan decomposition $g=\wt k_ga_gk_g\in KA^+K.$ 
Here, $a_g=\diag(\sigma_1(g),\cdots,\sigma_n(g))$ where $\sigma_1(g)\geqslant\cdots\geqslant\sigma_n(g)$ are singular values of $g$ as the notation before.
The Cartan projection of $g$ is defined to be $\kappa(g)\defeq \diag(\log\sigma_1(g),\cdots\log\sigma_n(g))\in \frak a^+.$
Let $e_1,\cdots,e_n$ be the standard orthonormal basis of $\R^n$ and $E_i=\RR e_i\in\PP(\RR^n)$ be the corresponding point in the projective space.
We consider the following notions as in our first paper \cite{LPX}:
\begin{itemize}
	\item  $V_g^+\defeq\wt k_gE_1 \in \P(\R^n)$, which is an attracting point of $g$;
    \item $H_g^-\defeq  k_g^{-1}(E_2\oplus\cdots\oplus E_n)\subset \P(\R^d)$, which is a repelling hyperplane of $g$
 
 	(if $\sigma_1>\sigma_2$, then $V_g^+$ and $H_g^-$ are uniquely defined);

	\item $b(g^-,\epsilon)\defeq\{x\in\P(\R^n):d(x,H_g^-)>\epsilon \}$ for any $\epsilon>0$;

	\item $B(g^+,\epsilon)\defeq\{x\in\P(\R^n): d(x,V_g^+)\leq \epsilon \}$ for any $\epsilon>0$.
\end{itemize}

We have the following useful lemma \cite[Lemma 14.2]{benoist_random_2016} for later use.
\begin{lem}\label{lem:gv d v g-}
For any $g\in \SL_n(\R)$ and $V=\R v\in \P(\R^n)$, we have
\begin{align*}
  d(V, H_g^-)\leq \frac{\|gv\|}{\|g\|\|v\|}\leq  d(V, H_g^-)+\frac{\sigma_2(g)}{\sigma_1(g)} ,\quad
 d(gV,V_g^+)d(V, H_g^-)\leq \frac{\sigma_2(g)}{\sigma_1(g)}.
\end{align*}
\end{lem}


Let $\nu$ be a finitely supported probability measure on $\SL_n(\RR).$
It induces random walks on the projective space and the flag variety. Recall the flag variety of $\RR^n$ as 
\[\cF(\RR^n)\defeq\lb{\,\xi=(\xi^1\sbs\xi^2\sbs \cdots \sbs \xi^i\sbs\cdots\sbs \xi^{n-1}): \xi^i\text{ is a linear subspace of }\RR^n \text{ of dimension }i\,}\]
and that $\SL_n(\R)$ acts on $\cF(\RR^n)$ canonically.

We denote $G_\nu$ to be the group generated by $\supp\nu.$
If we further assume that $G_\nu$ is Zariski dense in $\SL_n(\RR).$ Then the Lyapunov spectrum $\lambda(\nu)=\lb{\lambda_1(\nu)\geqslant\cdots \geqslant\lambda_n(\nu)}$ is simple. Moreover, the random walks induced by $\nu$ on the projective space $\PP(\RR^n)$ and the flag variety $\cF(\RR^n)$ both have a unique stationary measure (\cite{furstenberg_noncommuting_1963,guivarch_frontiere_1985,goldsheid1989lyapunov}, see also \cite[Chapter 10]{benoist_random_2016}).


\subsection{Different notions of entropies}
The Furstenberg entropy is  mysterious and might be difficult to compute. 
We recall another notion of the entropy associated to the random walk. The \textit{random walk entropy} of $\nu$ is 
            \begin{equation}\label{eqn: random walk entropy}
                h_{\mr{RW}}(\nu)=\lim_{k\to\infty}\frac{1}{k}H(\nu^{*k}).
            \end{equation}

\begin{rem}
In \cite{barany_hausdorff_2017} and \cite{rapaport_exact_2021}, the notion of the entropy they used is the following: for a discrete measure $\nu=\sum p_i\delta_{g_i}$, $H(\nu)\defeq H(p)=- \sum p_i\log p_i$. This notion works well in many settings of IFSs. For general semigroup actions, the random walk entropy $h_{\mr{RW}}(\nu)$ as in \cite{hochman_dimension_2017} is more precise. 
In particular, if $\supp\nu$ freely generates a free semigroup then $H(\nu)=h_{\mr{RW}}(\nu)$. 
This kind of entropy was first studied by Avez in \cite{Avez} to study the structure of the group action on the boundary. 

\end{rem}

To obtain a more calculable dimension formula, it is expected to show that the Furstenberg entropy in \cref{thm:lyapunov} is equal to the random walk entropy. In the following proposition we will see that for some concrete examples, we do have this equality. We say a representation $\rho:\Gamma\to\SL_n(\RR)$ is Zariski dense if $\rho(\Gamma)$ is Zariski dense.


\begin{prop}\label{prop:equal.entropy.geometry}
     Let $\Gamma$ be a hyperbolic group and $\rho:\Gamma\rightarrow\SL_n(\R)$ be a Zariski dense Borel Anosov representation. 
     Let $\nu$ be a finitely supported probability measure on $\rho(\Gamma)$ such that $G_\nu$ is Zariski dense. Then the unique $\nu$-stationary measure $\mu$ on $\P(\R^n)$ satisfies
    \[h_{\mathrm{F}}(\mu,\nu)=h_{\mr{RW}}(\nu).  \]
\end{prop}
\begin{rem}
    The Zariski density assumptions on $\rho$ and $G_\nu$ are both nonnecessary. It is shown in \cite{LedLe232} that for every $\nu$ supported on $\rho(\Gamma)$ with non-elementary $G_\nu$, the equality of entropies holds.
\end{rem}

To show Propostion \ref{prop:equal.entropy.geometry}, we may consider the random walk on the flag variety. 


\color{black}

\begin{prop} \label{prop.entropy.discrete}
Let $\nu$ be a probability measure on $\SL_n(\R)$ such that $G_\nu$ is a Zariski dense discrete subgroup. Let $\mu_\cF=\mu_{\cF}(\nu)$ and $\mu=\mu(\nu)$ be the unique $\nu$-stationary measure on $\cF(\RR^n)$ and $\P(\R^n)$ respectively. Then we have 
\[h_{\mr{RW}}(\nu)=h_{\mathrm{F}}(\mu_{\mathcal F},\nu)\geq h_{\mathrm{F}}(\mu,\nu).\]

\end{prop}

\begin{proof}
Proposition \ref{prop.entropy.discrete} follows from \cite[Theorem 2.31]{furman_random_2002} (originally in \cite[Theorem 3.2]{kaimanovich_random_1983}, \cite[Section 3.2]{ledrappier_poisson_1985}). In order to apply their result, we need to use \cite[Theorem 2.21]{furman_random_2002} (originally in \cite{kaimanovich_random_1983}, \cite{ledrappier_poisson_1985}) to obtain that $(\supp\mu_\cF,\mu_\cF)$ is the Poisson boundary of $(G_\nu,\nu)$. Then the Furstenberg entropy of the Poisson boundary is exactly the random walk entropy for discrete $G_\nu$ due to \cite[Theorem 2.31]{furman_random_2002}.
\end{proof}

\begin{proof}[Proof of \cref{prop:equal.entropy.geometry}]
By Proposition \ref{prop.entropy.discrete} and that the image of an Anosov representation is discrete, it remains to prove that $h_{\mathrm{F}}(\mu_\calF,\nu)=h_{\mathrm{F}}(\mu,\nu)$.

    Consider the canonical projection 
    \[\pi:\calF(\RR^n)\to \P(\R^n),~\xi=(\xi^1\sbs \xi^2\sbs\cdots\sbs\xi^{n-1})\mapsto \xi^1.\]
    Due to uniqueness of the Furstenberg measure on $\P(\R^n)$, we know that $\pi_*(\mu_\calF)=\mu$. By classical Rokhlin's disintegration  theorem, we can define a desintegration $\{\mu^\xi\}$ of the measure $\mu_\calF$ over $\mu$, where $\mu^\xi$ is a well-defined probabilty measure on $\pi^{-1}(\xi)$ for $
    \mu$ a.e. $\xi$.
    
    
    Let $L_\calF(\rho (\Gamma))$ and $L(\rho(\Gamma))$ be the limit sets on $\calF(\RR^n)$ and $\P(\R^n)$ (the closure of attracting fixed points of proximal elements) respectively. \cite[Lemma 9.5]{benoist_random_2016} tells us that $L_\calF(\rho (\Gamma))$ (resp. $L(\rho(\Gamma))$) is the unique $\rho(\Gamma)$-minimal closed invariant subset on $\calF(\RR^n)$ (resp. $\P(\R^n)$).
    Hence $\mu_\calF$ (resp. $\mu$) is supported on the limit set $L_\calF(\rho(\Gamma))$ (resp. $L(\rho(\Gamma))$).
    Before continuing, we need a lemma about the structure of the limit sets.
{\color{magenta} }
\begin{lem}\label{lem.trivial.fibre}
    Let $\rho:\Gamma\rightarrow \SL_n(\R)$ be a Borel Anosov representation. Then the canonical projection $\pi: L_\calF(\rho(\Gamma))\rightarrow L(\rho(\Gamma))$ has a trivial fibre.
\end{lem}
\begin{proof}
    Because $\rho:\Gamma\rightarrow \SL_n(\RR)$ is Borel Anosov, we can apply \cite[Theorem 31.1]{canary_note}. Then there exists a continuous $\rho$ equivariant map $(\xi^1,\xi') $ from $\partial\Gamma$ to $\P(\R^n)\times \mathrm{Grass}(n-1, \R^n)$, such that $\xi^1$ satisfies the Cartan property in \cite[Chapter 30]{canary_note} and $\xi^1(x)\sbs\xi'(x),$
    \[\xi^1(x)\oplus \xi'(y)=\R^n,\quad \forall x\ne y. \]
    It follows that $\xi^1$ is injective from $\partial\Gamma$ to $\P(\R^n)$. Moreover, the image of $\xi^1$ is exactly $L(\rho(\Gamma)).$ 
    
    From Borel Anosov property, $\xi^1$ can be extended to  a limit map $\xi:\partial\Gamma\rightarrow \calF $ given by $\xi(x)=(\xi^1(x)\subset \xi^2(x)\subset \cdots\subset \xi^{n-1}(x))$. \footnote{The existence of the limit map $\xi^k(x)$ is from $P_k$-Anosov and the consistence condition $\xi^k(x)\subset \xi^{k+1}(x)$ can be deduced from the Cartan property of the limit map (see for example Section 30 and 31 in \cite{canary_note}), that is $\xi^k(x)=\lim_{n\rightarrow\infty}U_k(\gamma_n) $. Here $\gamma_n$ is a sequence converges to the boundary point $x$ and $U_k(\gamma_n)=k_{\gamma_n}(E_1\oplus\cdots \oplus E_k)$ from the Cartan decomposition $\gamma_n=k_{\gamma_n}a_{\gamma_n}k'_{\gamma_n}$. } The image of $\xi$ is $L_\calF(\rho(\Gamma))\subset \calF$. The map $\xi^1$ being injective implies that the natural projection of $L_\calF(\rho(\Gamma))$ to $L(\rho(\Gamma))$ has trivial fiber. 
\end{proof}

    By Lemma \ref{lem.trivial.fibre}, we know that the fiber is trivial for all $\xi\in L(\rho(\Gamma))$. Hence we have the relation of \textit{relative measure-preserving} of $(\calF(\RR^n),\mu_\calF)\rightarrow (\P(\R^n),\mu )$, that is for $\nu$ a.e. $g$ and $\mu$ a.e. $\xi$
    \[  g\mu^\xi=g\delta_{\pi^{-1}\xi}=\delta_{\pi^{-1}(g\xi)}=\mu^{g\xi}. \]
    By \cite[Proposition 2.25]{furman_random_2002} (see also \cite{kaimanovich_random_1983}), we obtain 
    \[h_{\mathrm{F}}(\mu_\calF,\nu)=h_{\mathrm{F}}(\mu,\nu) \]
    Then by Proposition \ref{prop.entropy.discrete}, the proof is complete.
\end{proof}

\section{Free sub-semigroups in hyperbolic groups}\label{sec: geometric group thry}
This section is devoted to establish the geometric group theoretical preparation for the later proof of variational principle. In this section, $\Gamma$ is a finitely generated group with a fixed symmetric generating set $\cS$. For every $g\in \Gamma,$ let $|g|$ denote the word length of $g$ with respect to $\cS.$ For a positive integer $L,$ let $A(L)\defeq\lb{g\in\Gamma: |g|=L}$ refer to an annulus. The following is the main proposition of this section.

\begin{prop}\label{prop: free semigroup}
    Let $\Gamma$ be a non-elementary torsion-free hyperbolic group and $\rho: \Gamma\to\SL_n(\RR)$ be a faithful representation. Let $\bG$ be the Zariski closure of $\rho(\Gamma)$ which is assumed to be Zariski connected. Then there exists a finite subset $F\sbs\Gamma$ with $\#F\geqslant 3,$ constants $C_1,C_2,L_0>0$ and $m\in\ZZ_+$ such that the following holds.
    
    For every subset $S\sbs A(L)$ for some $L\geqslant L_0$ there exists a subset $S'\sbs S$ with $\#S'\geqslant C_1^{-1} \#S$ and $F'\sbs F$ with $\#F'=\#F-2$ satisfying
    \begin{enumerate}[(1)]
        \item $\lb{\rho(f)^m:f\in F'}$ generates a semigroup whose Zarski closure is $\bG.$
        \item $\wt S\defeq\lb{sf^{\varsigma }:s\in S',f\in F',\varsigma = m,2m}\sbs\Gamma$ freely generates a free semigroup.
        \item For any sequence of elements $\wt s_1,\cdots,\wt s_k \in \wt S,$ we have
	\begin{equation}\label{eqn: constant cancellation2}
		|\wt s_1\cdots\wt s_k|\geqslant \sum_{i=1}^{k}|\wt s_i|-kC_2.
	\end{equation}
    \end{enumerate}    
\end{prop}

\subsection{Preliminaries on geometric group theory}

Recall that $\Gamma$ is a finitely generated group and $\cS$ is a fixed symmetric generating set. Let $X=\Cay(\Gamma,\cS)$ be the Cayley graph of $\Gamma$ with respect to $\cS.$ Endow $X$ with the graph metric $d_{\cS},$ which makes $X$ a proper geodesic space. We abbreviate $d_\cS$ to $d$ in this section. Then $\Gamma$ has a natural left action on $(X,d)$ by isometries. Letting $o$ be the point in $X$ corresponding to the identity in $\Gamma,$ we fix $o$ as the base point. Then the word length $|g|$ is equal to $d(o,go).$

For every subset $Y\sbs X$ and $r>0,$ we use $\cN_r(Y)$ to denote the $r$-neighborhood of $Y.$ For two subsets $Y_1,Y_2\sbs X,$ we use $d_{\mr H}(Y_1,Y_2)$ to denote the Hausdorff distance between $Y_1,Y_2$ with respect to $d,$ given by
\[d_{\mr H}(Y_1,Y_2)\defeq\inf\lb{0<r\leqslant\infty: Y_1\sbs \cN_r(Y_2)\text{ and }Y_2\sbs \cN_r(Y_1)}.\]

For a subset $Y\sbs X$ and a point $x\in X,$ we denote $\pi_Y(x)$ to be the projection of $x$ on $Y,$ that is $\pi_Y(x)\defeq\lb{y\in Y:d(x,y)=d(x,Y)}.$ For another subset $Z\sbs X,$ the projection of $Z$ on $Y$ is $\pi_Y(Z)\defeq\cup_{z\in Z}\pi_Y(z).$

For a rectifiable path $\alpha\sbs X,$ we denote $\alpha_-$ and $\alpha_+$ to be the initial and terminal points of $\alpha,$ respectively. The length of $\alpha$ is denoted by $\ell(\alpha).$ For every pair $x,y\in X,$ we denote $[x,y]$ to be a choice of a geodesic between $x$ and $y.$

A path $\alpha$ is called a \textit{$c$-quasi-geodesic} for $c\geqslant 1$ if
\[\ell(\beta)\leqslant c\cdot d(\beta_-,\beta_+)+c \]
for every rectifiable subpath $\beta\sbs\alpha.$ Morse lemma states that every $c$-quasi-geodesic $\alpha$ is contained in the $c'$-neighborhood of $[\alpha_-,\alpha_+]$ in a $\delta$-hyperbolic space, where $c'$ only depends on $c$ and $\delta.$ A subset $Y\sbs X$ is called \textit{$c$-quasi-convex} for $c\geqslant 0$ if for every $x_1,x_2\in Y,$ we have $[x_1,x_2]\sbs\cN_c(Y).$ A quasi-geodesic can also be interpreted as a quasi-convex subset.

Recall that $\Gamma$ is a ($\delta$-)hyperbolic group if and only if $X$ is a ($\delta$-)hyperbolic space.
That is, for every geodesic triangle in $X,$ every edge is contained in the $\delta$-neighborhood of the other two edges. In the following of this section, we always assume that $\Gamma$ is a finitely generated $\delta$-hyperbolic group. 
Every infinite order element $g\in\Gamma$ is \textit{loxodromic} in the following sence
\[\lim_{n\to+\infty}\frac{1}{n}|g^n|=\lim_{n\to+\infty}\frac{1}{n}d(o,g^no)>0.\]
In fact, for every loxodromic element $g,$ the map $n\mapsto g^no$ is a \textit{quasi-isometric embedding} from $\ZZ$ to $(X,d).$ That is, there exists $c_1,c_2,c_3,c_4>0$ such that for every $m,n\in\ZZ,$
\[c_1|m-n|-c_2\leqslant d(g^mo,g^no)\leqslant c_3|m-n|+c_4.\]

For every loxodromic element $g\in\Gamma,$ the set
\[E(g)\defeq\lb{h\in\Gamma:d_{\mr H}(h\pair{g}o,\pair{g}o)<\infty}\]
is a subgroup of $\Gamma$ satisfying $[E(g):\pair g]<\infty.$
We use $\Ax(g)=E(g)o\sbs X$ to denote the axis corresponding to $g,$ which is a quasi convex subset.
We also remark that if $g_1\in E(g_2)$ is also loxodromic then $E(g_1)=E(g_2).$
Two loxodromic elements $g_1,g_2$ are called \textit{independent}\footnote{This definition of independence is different with the one in \cite{Yang19}. But this is enough for our applications for hyperbolic groups.
} if $E(g_1)\ne E(g_2),$ namely $E(g_1)\cap E(g_2)$ is finite. Moreover, we have the following bounded intersection property.
\begin{lem}\label{lem: bounded intersection}
 	Let $\Ax(g_1),\Ax(g_2)\sbs X$ be different axes, then for every $r>0,$ $\cN_r(\Ax(g_1))\cap\cN_r(\Ax(g_2))$ is bounded.
\end{lem}
\begin{proof}
	Since $d_{\mr H}(\Ax(g_i),\pair{g_i}o)<\infty,$ it suffices to show that $ \cN_r(\pair{g_1}o))\cap \cN_r(\pair{g_2}o)$ is bounded. 
	Otherwise, there exists infinitely many pairs of integers $(n_1,n_2)$ such that $d(g_1^{n_1}o,g_2^{n_2}o)\leqslant 2r.$ Then there exists $(n_1,n_2)\ne(n_1',n_2')$ such that
	\[g_2^{-n_2}g_1^{n_1}=g_2^{-n_2'}g_1^{n_1'}.\]
	This implies that $g_1^{n_1'-n_1}=g_2^{n_2'-n_2}$, which contradicts $\Ax(g_1)\ne\Ax(g_2).$
\end{proof}

In the case of $\Gamma$ is torsion-free, every nontrivial element is loxodromic. By a classification of virtually cyclic group \cite[Lemma 11.4]{JH04}, $E(g)$ is cyclic for every nontrivial element $g\in\Gamma.$ Moreover, $g_1$ and $g_2$ are independent if and only if $E(g_1)\cap E(g_2)=\lb\id.$
\subsection{An extension lemma}

Now we give the main technical tool for showing Proposition \ref{prop: free semigroup}.
It is a variant of \cite[Lemma 2.19]{Yang19}, which states more generally for the group with contracting elements. 
A direct proof for the case of hyperbolic groups is given in this section, which is also inspired by the work of W. Yang.

For a subset $S\sbs\Gamma,$ we say $S$ is \textit{$R$-separated} for $R>0$ if $\lb{so:s\in S}$ is $R$-separated in $X.$

\begin{prop}\label{prop: extension lemma}
	Let $\Gamma$ be a non-elementary hyperbolic group. Let $F\sbs \Gamma$ be a finite subset of pairwise independent loxodromic elements with $\#F\geqslant 3.$ Then there exist $m_0,R,C_1,L_0>0$ such that the following holds.
	
	For every $R$-separated subset $S\sbs A(L)$ for some $L\geqslant L_0$ there exists a subset $S'\sbs S$ with $\#S'\geqslant C_1^{-1} \#S$ and $F'\sbs F$ with $\#F'=\#F-2$ such that for every $m\geqslant m_0,$
	\[\wt S\defeq\lb{sf^{\varsigma }:s\in S',f\in F',\varsigma = m,2m}\sbs\Gamma\]
	freely generates a free semigroup. Furthermore, there exists $C_2>0$ only depends on $F$ such that for any sequence of elements $\wt s_1,\cdots,\wt s_k \in \wt S,$ we have
	\begin{equation}\label{eqn: constant cancellation}
		|\wt s_1\cdots\wt s_k|\geqslant \sum_{i=1}^{k}|\wt s_i|-kC_2.
	\end{equation}
\end{prop}

Before going through the proof, we need some preparation on hyperbolic groups. Recall that $X$ is a $\delta$-hyperbolic space. Let $\pair{x,y}_z=(d(x,z)+d(y,z)-d(x,y))/2$ be the Gromov product, then $d(z,[x,y])-10\delta\leqslant \pair{x,y}_z\leqslant d(z,[x,y]).$
\begin{defi}\label{def: CD-chain}
	A $(\tau,D)$-chain is a sequence of points $x_0,x_1,\cdots,x_n\in X$ such that 
	\begin{itemize}
		\item $\pair{x_{i-1},x_{i+1}}_{x_i}\leqslant \tau$ for every $1\leqslant i\leqslant n-1,$ and
		\item $d(x_{i-1},x_i)\geqslant D$ for every $1\leqslant i\leqslant n.$
	\end{itemize}
\end{defi}
By connecting two consecutive points in a $(\tau,D),$ we obtain a quasi-geodesic. The following lemma shows such path is indeed a uniform quasi-geodesic. See also \cite[Section 3B]{Seb22}.

\begin{lem}\label{lem: CD-chain}
	For every $\tau>0,$ there exists $D=D(\tau),c=c(\tau)>0$ such that for every $(\tau,D)$-chain $x_0,x_1,\cdots,x_n,$ the path $\bigcup_{i=1}^n[x_{i-1},x_i]$ is a $c$-quasi-geodesic.
\end{lem}

The following it a bounded projection property for different axes. A general version for spaces with contracting elements can be found in \cite[Lemma 2.17]{Yang19}.

\begin{lem}\label{lem: bounded projection}
	Let $f_1,f_2\in\Gamma$ be independent loxodromic elements. There exists $\tau_1>0$ such that for every $g\in\Gamma,$ we have
	\[\min\lb{d(o,\pi_{\Ax(f_1)}(go)),d(o,\pi_{\Ax(f_2)}(go))}\leqslant \tau_1. \]
\end{lem}
\begin{proof}
	Let $z_i\in \pi_{\Ax(f_i)}(go),$ we first show that $[o,z_i][z_i,go]$ is a $c_1$-quasi-geodesic for some $c_1>0$ only depends on $\Ax(f_1)$ and $\Ax(f_2).$ Note that there exists $c_2>0$ such that $\Ax(f_i)$ is $c_2$-quasi-convex.
	Fixing an $i\in\lb{1,2},$ for every $x\in [o,z_i]$ and $y\in[z_i,go],$ we have 
	\[d(y,x)\geqslant d(y,\Ax(f_i))-c_2= d(y,z_i)-c_2.\]
	Hence $d(x,z_i)+d(z_i,y)\leqslant d(x,y)+2d(z_i,y)\leqslant 3d(x,y)+2c_2$ is a $(3+2c_2)$-quasi-geodesic.
	
	By Morse lemma, $[o,z_i]\sbs\cN_c([o,go])$ where $c=c(c_1)>0.$ If both $d(z_1,o)$ and $d(z_2,o)$ are larger than $\tau_1,$ we can choose $z_i'\in[o,z_i]$ with $d(o,z_i')=\tau_1.$ Let $w_i\in[o,go]$ with $d(w_i,z_i')\leqslant c,$ then $|d(w_i,o)-\tau_1|\leqslant c.$ Hence $w_1,w_2\in \cN_{3c}(\Ax(f_1))\cap \cN_{3c}(\Ax(f_2)).$ Since $d(w_1,o)\geqslant \tau_1 -c,$ this contradicts Lemma \ref{lem: bounded intersection} for a sufficiently large $\tau_1.$
\end{proof}
\begin{defi}
	Let $g\in\Gamma$ and $\tau>0,$ we call a loxodromic element $f\in\Gamma$ is $\tau$-contracting for $g$ if for every $h\in E(f),$ we have $\pair{go,ho}_{o}\leqslant \tau$ and $\pair{g^{-1}o,ho}_{o}\leqslant \tau.$
\end{defi}

\begin{lem}
	Let $F\sbs\Gamma$ be a finite subset of pairwise independent loxodromic elements with $\# F\geqslant 3.$
	There exists $\tau>0$ such that for every $g\in\Gamma,$ there exists $F_g\sbs F$ with $\# F_g= \# F-2$ such that every element in $F_g$ is $\tau$-contracting for $g.$
\end{lem}
\begin{proof}
	By the previous lemma, there exists $\tau_1>0$ such that for every $f_1,f_2\in F$ and $ g\in\Gamma,$ we have $\min\lb{d(o,\pi_{\Ax(f_1)}(go)),d(o,\pi_{\Ax(f_2)}(go))}\leqslant \tau_1.$ Note that there exists $c_1>0$ such that for every $f\in F,$ $\Ax(f)$ is $c_1$-quasi-convex. We take $\tau=\tau_1+c_1+10\delta.$
	
	For each $g\in \Gamma$ and $f\in F$ satisfying $d(o,\pi_{\Ax(f)}(go))\leqslant \tau_1.$ For every $h\in E(f),$ we have $[o,ho]\in\cN_{c_1}(\Ax(f)).$ Hence 
	\[d(go,[o,ho])\geqslant d(go,\Ax(f))-c_1\geqslant d(go,o)- d(o,\pi_{\Ax(f)}(go))-c_1\geqslant d(o,go)-\tau_1-c_1.\]
	Then $\pair{o,ho}_{go}\geqslant d(go,[o,ho])-10\delta\geqslant  d(o,go)-\tau_1-c_1-10\delta=d(o,go)-\tau.$ We obtain 
 \[\pair{go,ho}_o=d(o,go)-\pair{o,ho}_{go}\leqslant \tau.\]
	
	To complete the proof, we apply the argument to both $g$ and $g^{-1}.$ By the previous lemma, there are at least $\# F-2$ elements in $F$ which are $\tau$-contracting for $g.$
\end{proof}

Now we are at the stage of proving Proposition \ref{prop: extension lemma}.
\begin{proof}[Proof of Proposition \ref{prop: extension lemma}.]
	We apply the previous lemma to $F$ and obtain a constant $\tau>0.$ Let $C_1=\#F(\#F-1)/2$. By the pigeonhole principle, there exists $F'\sbs F$ with $\#F'=\#F-2$ such that
	\[S'=\lb{s\in S:F_s=F'}\]
	has the cardinality at least $C_1^{-1}\# S,$ where $F_s\sbs F$ is the subset given by the previous lemma consisting of $\tau$-contracting elements. 
	
	By Lemma \ref{lem: CD-chain}, there exists $D=D(\tau)>0$ and $c_1=c_1(\tau)>0$ such that every $(\tau,D)$-chain forms a $c_1$-quasi-geodesic. By Morse Lemma, every $c$-quasi-geodesic $\alpha$ is contained in $\cN_C([\alpha_-,\alpha_+])$ for some $C>0.$ Now we take $C_2=2C,R=4C+1,$ $L_0=D+4C.$ The choice of $m_0$ will be given later. Let $c_2>0$ such that every axis $\Ax(f)$ is $c_2$-quasi-convex for $f\in F.$ Now we show that $\wt S$ freely generates a free semigroup.
	
	Let $k\geqslant 1$ and $s_i\in S',f_i\in F',\varsigma_i\in\lb{m,2m}$ for $1\leqslant i\leqslant k.$ We consider
	\[x_i=s_1f_1^{\varsigma_1}\cdots s_if_i^{\varsigma_i}o,\quad y_i=s_1f_1^{\varsigma_1}\cdots s_i o.\]
	We claim that $x_0,y_1,x_1,\cdots,y_k,x_k$ is a $(\tau,D)$-chain. Assume that $m_0$ is large enough guaranteeing $|f^m|>D$ for every $f\in F$ and $m\geqslant m_0.$ Since $|s_i|=L\geqslant D+4C$ and $|f_i^{\varsigma_i}|\geqslant D$ for each $i,$ the second condition in Definition \ref{def: CD-chain} is verified. Besides, for each $i,$ we have
	\[\pair{x_{i-1},x_i}_{y_i}=\pair{s_i^{-1}o,f_i^{\varsigma_i}o}_o\leqslant \tau,\quad\pair{y_i,y_{i+1}}_{x_i}=\pair{f_i^{-\varsigma_i}o,s_{i+1}o}_o\leqslant \tau \]
	by the $\tau$-contracting property. 
 
 Hence $\alpha =[x_0,y_1][y_1,x_1]\cdots[x_{k-1},y_k][y_k,x_k]$ is a $c_1$-quasi geodesic, which is contained in the $C$-neighborhood of $[\alpha_-,\alpha_+]=[o,x_k].$ Then we have two estimates on the length of $d(o,x_k).$ Firstly, since $|s_i|\geqslant D+4C $ and $|f_i^{\varsigma_i}|\geqslant D,$ we have 
\begin{equation}\label{eqn: (tau,D)-chain}
    d(o,x_k)\geqslant \sum_{i=1}^k(d(x_{i-1},y_i)+d(y_i,x_i))-4kC\geqslant k(2D+2C)-4kC>0.
\end{equation}	
	This implies that $x_k\ne o.$ Besides, we have
	\[d(o,x_k)\geqslant \sum_{i=1}^k d(x_{i-1},x_i)-2kC=\sum_{i=1}^k d(x_{i-1},x_i)-kC_2,\]
	which gives the desired estimate \eqref{eqn: constant cancellation}.
	
	\paragraph{Checking the freeness.}
	We consider two such sequences $s_i, f_i,\varsigma_i$ for $1\leqslant i\leqslant k$ and $s_j',f_j',\varsigma_j'$ for $1\leqslant j\leqslant \ell.$ We get two sequences of points $x_i,y_i$ and $x_j',y_j'$ in $X.$ Assume that $x_k=x_\ell',$ we are going to show that $k=\ell$ and $s_i=s_i',f_i=f_i',$ $\varsigma_i=\varsigma_i'.$ By an inductive argument, it suffices to check for $i=1.$ Let $\alpha$ be a fixed geodesic connecting $o$ and $x_k=x_\ell'.$ Then there exist $z_1,z_1'\in\alpha$ such that $d(y_1,z_1),d(y_1',z_1')\leqslant C.$ Since $d(y_1,o)=d(y_1',o)=L,$ we have $d(z_1,z_1')=|d(o,z_1)-d(o,z_1')|\leqslant 2C.$ This implies that $d(y_1,y_1')\leqslant 4C.$ By the $R$-separation, $y_1=y_1'.$
	
	Now we consider a geodesic $\beta$ connecting $y_1=y_1'$ and $x_k=x_\ell'.$ Without loss of generality, we assume that $d(y_1,x_1)\leqslant d(y_1',x_1').$ 
	Take a point $w\in\pi_\beta(x_1)$ then $d(x_1,w)\leqslant C.$
	Letting $x'\in [y_1',x_1']$ with $d(y_1',x_1')=d(x_1,y_1),$ we have $d(x',w)\leqslant 3C$ since $[y_1',x_1']\in\cN_C(\beta).$ This implies $w\in\cN_{3C}([y_1,x_1])\cap\cN_{3C}([y_1',x_1']).$ 
	By the quasi-convexity of axes, we have $[y_1,x_1]\sbs s_1\cN_{c_2}(\Ax(f_1))$ and $[y_1',x_1']\sbs s_1\cN_{c_2}(\Ax(f_1')).$
	Therefore, both $y_1$ and $w$ are contained in $s_1(\cN_{3C+c_2}(\Ax(f_1))\cap \cN_{3C+c_2}(\Ax(f_1'))).$ 
	Note that $d(y_1,w)\geqslant d(y_1,x_1)-d(x_1,w)\geqslant |f_1^{\varsigma_1}|-C,$ we conclude that 
	\[\diam \cN_{3C+c_2}(\Ax(f_1))\cap \cN_{3C+c_2}(\Ax(f_1'))\geqslant d(y_1,w)\geqslant |f_1^{\varsigma_1}|-C.\]
	By Lemma \ref{lem: bounded intersection}, the $(3C+c_2)$-neighborhoods of axes of different elements in $F$ have uniformly bounded intersections.
	We take $m_0$ sufficiently large at beginning such that $|f^m|-C$ is strictly larger than diameters of all such intersections for every $f\in F,m\geqslant m_0$, this forces $f_1=f_1'.$
	
	Finally, we should show $\varsigma_1=\varsigma_1'.$ Otherwise, we assume that $\varsigma_1=m$ and $\varsigma_1'=2m.$ 
	In this case, we have 
	\[y_2=s_1f_1^ms_2 o,\quad x_1=s_1f_1^m o,\quad x_1'=s_1f_1^{2m}o ,\quad y_2'=s_1f_1^{2m}s_2'o.\] 
	By the $\tau$-contracting property, these points form a $(\tau,D)$-chain. This leads to that
	\[x_k,y_k,\cdots,y_2,x_1,x_1',y_2',\cdots,y_\ell',x_\ell'.\]
	is also a $(\tau,D)$-chain. Applying the estimate in \eqref{eqn: (tau,D)-chain}, this contradicts $x_k=x_\ell'.$
\end{proof}

\subsection{Proof of Proposition \ref{prop: free semigroup}.}
To prove Proposition \ref{prop: free semigroup}, we need to find a finite subset consisting of pairwise independent loxodromic elements satisfying the desired condition on Zariski closures.

\begin{lem}\label{lem: independent elements}
	Let $\Gamma$ be a non-elementary torsion-free hyperbolic group and $\rho: \Gamma\to\SL_n(\RR)$ be a faithful representation. Let $\bG$ be the Zariski closure of $\rho(\Gamma)$ and assume $\bG$ is Zariski connected. For every $k\geqslant 0,$ there exists a finite subset $F\sbs\Gamma$ consisting of pairwise independent nontrivial elements such that for every subset $F'\sbs F$ with $\# F'\geqslant F-k,$ $\rho(F')$ generates a semigroup whose Zarski closure is $\bG.$
\end{lem}



\begin{proof}

We will find $F=F_k$ for each $k\geqslant 0$ by an induction on $k.$

For the case of $k=0,$ applying \cite[Lemma 3.6]{morris_variational_2023}, we can find a finite subset $F_0\sbs \Gamma$ such that $\rho(F_0)$ generates a semigroup whose Zariski closure is $\bG.$ Furthermore, we can assume that elements in $F_0$ are pairwise independent by the following process. If $f_1,f_2\in F_0$ are not independent, then there exists $f\in \Gamma$ and $m_1,m_2\in\ZZ$ such that $f_i=f^{m_i}$ since $\Gamma$ is torsion free. We can remove $f_1,f_2$ and add $f$ in $F_0.$

Assume that $F_{k-1}$ is found. Recall that for every $f\in\Gamma,$ $E(f)$ is a cyclic group containing all elements which are not independent of $f.$
Notice that Zariski closure of $\rho(E(f))$ is commutative. 
On the other hand, applying Tits alternative for hyperbolic groups, $\Gamma$ contains a nonabelian free subgroup. This implies that $\rho(E(f))$ is not Zariski dense in $\bG.$ By the connectivity of $\bG,$ the set $S=\Gamma\sm\bigcup_{f\in F_{k-1}} E(f)$ is Zariski dense in $\bG.$ Applying \cite[Lemma 3.6]{morris_variational_2023} and the argument for $k=0$ case, we can find a finite subset $\wt F\sbs S$ consisting of pairwise independent elements such that $\rho(\wt F)$ generates a semigroup whose Zariski closure is $\bG.$ We take $F_k=F_{k-1}\cup \wt F,$ which gives a desired construction for the $k$-case.
\end{proof}

\begin{proof}[Proof of \cref{prop: free semigroup}]
     Applying Lemma \ref{lem: independent elements} to the case of $k=2,$ we can find a finite subset $F\sbs\Gamma$ consisting of pairwise independent elements such that for every $F'\sbs F$ with $\#F'=\#F-2,$ $\rho(F')$ generates a semigroup whose Zariski closure is $\bG.$ We apply Proposition \ref{prop: extension lemma} to $F.$ Let $m_0$ be the constant given by this proposition.
    In order to find $m\geqslant m_0$ satisfying the first condition in the proposition, we need the following lemma.
	\begin{lem}\label{lem: Zariski closure 2}
		For every $g\in \SL_n(\R),$ there exists a positive integer $\ell$ such that for every $m$ coprime with $\ell,$ the Zariski closure of $\pair{g^m}$ is equal to the Zariski closure of $\pair{g}.$
	\end{lem}
	\begin{proof}
		Let $\bH$ be the Zariski closure of $\pair{g}.$ Let $\ell$ be the number of connected components of $\bH.$ Let $\bH'$ be the Zariski closure of $\pair{g^m}$ where $m$ is coprime with $\ell.$ Then $\bH'$ is a finite index algebraic subgroup of $\bH$ and hence a union of some connected components of $\bH.$ Note that the action of $\pair{g}$ given by left translation is transitive among connected components of $\bH$ and so does $\pair{g^m}$, due to $m$ coprime with $\ell$. Hence $\bH'=\bH.$
	\end{proof}
    We take $m\geqslant m_0$ to be a sufficiently large prime number such that for every $f\in F,$ the Zariski closure of $\pair{\rho(f)^m}$ equals to the Zariski closure of $\pair{\rho(f)}.$ Then for every $F'\sbs F$ with $\#F'=\# F-2,$ the Zariski closure of the semigroup generated by $\lb{\rho(f)^m:f\in F'}$ equals to that of $\rho(F),$ which is $\bG.$ The first condition holds.

    Note that for every finite subset $S\sbs \Gamma,$ it contains an $R$-separated subset of cardinality at least $(\#\cS+1)^{-R}\#S,$ where $\cS$ is the symmetric generating set of $\Gamma.$ Then the last two conditions follow from Proposition \ref{prop: extension lemma} by enlarging $C_1$ suitably.
\end{proof}

\section{Variational principle for Anosov representations 
}\label{sec:limitset}

In this section, we will show a variational principle for dimensions of limit sets of Borel Anosov representations. Let $\Gamma$ be a finitely generated hyperbolic group with a fixed finite symmetric generating set $\cS\sbs\Gamma.$ Let $\rho:\Gamma\to \SL_n(\RR)$ be an irreducible Borel Anosov representation. 
Recall that $\kappa(\rho(\gamma))=\diag(\log\sigma_1(\rho(\gamma)),\cdots,\log\sigma_n(\rho(\gamma)))$ is the Cartan projection of $\rho(\gamma).$

\subsection{The variational principle to positive linear functions on \texorpdfstring{$\frak a^+$}{a\^{}+}}
Recall the Cartan algebra $\frak a=\{\lambda=\diag(\lambda_1,\cdots, \lambda_n):\lambda_i\in\R,
\ \sum\lambda_i=0\} $ and a positive Weyl chamber $\frak a^+$ mentioned in Section \ref{subsec: actions}. Let $\psi$ be a linear function on $\frak a$ which is positive with respect to $\frak a^+.$ Specifically, $\psi$ can be expressed as
\[\psi(\lambda)=\sum_{i=1}^{n-1}a_i\cdot \alpha_i(\lambda),\]
where $a_1,\cdots,a_{n-1}\geqslant 0$ are not all zero and $\alpha_i(\lambda)=\lambda_i-\lambda_{i+1}$ are simple roots.

Recall that for a finitely supported probability measure $\nu$ on $\SL_n(\RR),$ the Lyapunov spectrum of $\nu$ is $\lambda(\nu)=(\lambda_1(\nu),\cdots,\lambda_n(\nu)).$ We also view $\lambda(\nu)$ as an element in $\frak a^+$ using the isomorphism $\frak a\cong \RR^n$ and $\psi$ acts on $\lambda(\nu).$ Then $\psi(\lambda(\nu))$ is a nonnegative number.

Throughout this subsection, we further assume that $\Gamma$ is torsion-free and $\rho$ is faithful. Note that $\Gamma$ is non-cyclic and hence non-elementary since $\rho$ is irreducible. 


\begin{prop}\label{prop: series variation}
	Let $\bG$ be the Zariski closure of $\rho(\Gamma)$ which is assumed to be Zariski connected. Let $\psi$ be a linear function as above. If the series $\sum_{\gamma\in\Gamma} \exp(-\psi(\kappa(\rho(\gamma))))$ diverges, then there exists $c>0$ such that the following holds. For every $\ve>0,$ there exists infinitely many positive integers $N$ with a finitely supported probability measure $\nu$ on $\rho(\Gamma)$ such that
	\begin{itemize}
		\item $G_\nu$ is Zariski dense in $\bG.$
		\item $\lambda_p(\nu)-\lambda_{p+1}(\nu)\geqslant cN$ for every $p=1,\cdots,n-1 .$
		\item $h_{\mr{RW}}(\nu)\geqslant (1-\ve)N$ and $\psi(\lambda(\nu))\leqslant (1+\ve)N.$
	\end{itemize}
\end{prop}

We first recall an estimate on the lost of singular values under composition. The following lemma is a direct consequence of combining Lemmas 2.5 and A.7 in \cite{bochi_anosov_2019}.
\begin{lem}\label{lem: almost additivity}
	Let $1\leqslant p\leqslant n.$ Given $c>0,$ then there exists $\delta>0$ such that the following holds. Let $(\gamma_k)_{k\in\NN}\in\cS^\NN$ satisfying for every $\ell\leqslant m,$ 
	\[\frac{\sigma_p(\rho(\gamma_{\ell+1}\cdots\gamma_m))}{\sigma_{p+1}(\rho(\gamma_{\ell+1}\cdots\gamma_m)) }\geqslant c\cdot e^{c(m-\ell)}.\]
	Then for every $\ell\leqslant k\leqslant m,$ we have
	\[\sigma_p(\rho(\gamma_{\ell+1}\cdots \gamma_m))\geqslant \delta \cdot \sigma_p(\rho(\gamma_{\ell+1}\cdots \gamma_k))\sigma_p(\rho(\gamma_{k+1}\cdots \gamma_m)) ,\]
	\[\sigma_{p+1}(\rho(\gamma_{\ell+1}\cdots \gamma_m))\leqslant \delta^{-1} \cdot \sigma_{p+1}(\rho(\gamma_{\ell+1}\cdots \gamma_k))\sigma_{p+1}(\rho(\gamma_{k+1}\cdots \gamma_m)) .\]
\end{lem}

\begin{proof}[Proof of Proposition \ref{prop: series variation}]

    Applying Proposition \ref{prop: free semigroup}, we obtain a finite subset $F\sbs\Gamma$ with $\# F\geqslant 3,$ constants $C_1,C_2,L_0>0$ and a positive integer $m.$ 
    For every $\ve>0$ sufficiently small, there are infinitely many integers $N$ such that 
	\[S_1=\lb{\gamma\in\Gamma:\psi(\kappa(\rho(\gamma)))\leqslant N} \]
	has cardinality at least $e^{(1-\ve )N}$ due to the divergence of the series. Since $\psi$ is positive, we have $\psi(\kappa(\rho(\gamma)))\geqslant c_1 (\log\sigma_p(\rho(\gamma))-\log\sigma_{p+1}(\rho(\gamma)))$ for some $1\leqslant p\leqslant n-1$ and $c_1>0.$ Because $\rho$ is Borel Anosov, there exists $c_2>0$ such that for $|\gamma|$ large enough,
	\[\psi(\kappa(\rho(\gamma)))\geqslant c_1 (\log\sigma_p(\rho(\gamma))-\log\sigma_{p+1}(\rho(\gamma)))\geqslant c_2|\gamma |.\]
	Hence $S_1\sbs\lb{\gamma\in\Gamma:|\gamma|\leqslant c_2^{-1}N}.$ Let $c_3=(2\#\log\cS)^{-1},$ then 
	\[\#\lb{\gamma\in\Gamma:|\gamma|\leqslant c_3N}\leqslant 2(\#\cS)^{c_3N}\leqslant \frac{1}{2} e^{(1-\ve)N}\]
	providing $\ve$ small and $N$ large. Hence there exists $c_3N\leqslant L\leqslant c_2^{-1}N$ such that
	\[S_2\defeq\lb{\gamma\in S_1: |\gamma|=L}\]
	has cardinality at least $c_2(2N)^{-1}e^{(1-\ve)N}\geqslant e^{(1-2\ve)N}$ assuming $N$ large. 
	By Proposition \ref{prop: free semigroup} , there exists $ S_3\sbs  S_2$ and $F'\sbs F$ with $\# S_3\geqslant C_1^{-1}\# S_2$ and $\#F'=\#F-2$ such that
	\[\wt S\defeq\lb{sf^\varsigma :s\in S_3,f\in F',\varsigma=m,2m }\]
	freely generates a free semigroup. Assuming $N$ large, we have $\#\wt S\geqslant \# S_3\geqslant e^{(1-3\ve )N}.$ Letting $\nu$ be the uniform measure on $\rho(\wt S),$ we now verify that this is a desired construction.
	\begin{itemize}
		\item Note that $\rho(f)^m=\rho(sf^m)^{-1}\rho(sf^{2m})\in G_\nu$ for every $f\in F'.$ By the first condition in \ref{prop: free semigroup}, we have $G_\nu$ is Zariski dense in $\bG.$
		\item By the definition of Borel Anosov representations, we can take $c_4>0$ such that for every $1\leqslant p\leqslant n-1$ and $|\gamma|$ large enough, $\log\sigma_p(\rho(\gamma))-\log\sigma_{p+1}(\rho(\gamma))\geqslant c_4|\gamma |$. 
        Assuming $N$ is large enough only depends on $m,$ for every $\wt s\in\wt S$ we have 
		\[|\wt s|\geqslant c_3N -2m\max_{f\in F}|f|\geqslant c_3'N+C_2\quad \text{and}\quad  |\wt s|\leqslant c_2^{-1}N +2m\max_{f\in F}|f|\leqslant c_2'^{-1}N\]
		where $c_2'=c_2/2,c_3'=c_3/2$ and $C_2$ is the constant given by Proposition \ref{prop: free semigroup}. Then for every $\wt s_1,\cdots \wt s_k \in\wt S,$ due to \eqref{eqn: constant cancellation2}, we have $|\wt s_1\cdots \wt s_k|\geqslant c_3' kN.$
		Hence for every $1\leqslant p\leqslant n-1 ,$ we have
		\begin{equation}\label{eqn: quasi-geodesic}
			(\log \sigma_p-\log\sigma_{p+1})(\rho(\wt s_1\cdots\wt s_k))\geqslant c_4 c_3'kN\geqslant c_2'c_3'c_4 \sum_{i=1}^k |\wt s_i|.
		\end{equation}
		Recall that $\supp\nu=\wt S$. By the first inequality in \eqref{eqn: quasi-geodesic} we have
		\[\lambda_p(\nu)-\lambda_{p+1}(\nu)=\lim_{k\to\infty}\frac{1}{k}\int [\log\sigma_p(\rho(\gamma))-\log\sigma_{p+1}(\rho(\gamma))]\dr\nu^{*k}(\gamma)\geqslant c_4c_3'N.\]
		Taking $c=c_4c_3',$ we obtain the conclusion.
		\item Since $\wt S$ freely generates a free semigroup and $\nu$ is the uniform measure on $\wt S,$ we have
		\[h_{\mr{RW}}(\nu)=\log\#\wt S\geqslant (1-3\ve )N.\]
		Shrinking $\ve >0,$ we obtain the first estimate.
		 
		In order to estimate $\psi(\lambda(\nu)),$ we need an almost additivity property of $\log\sigma_p.$ 
		Recall the second inequality in \eqref{eqn: quasi-geodesic} and the constant cancellation property in \eqref{eqn: constant cancellation2}, which verifies the condition of \cref{lem: almost additivity}. 
		By applying Lemma \ref{lem: almost additivity}, there exists $\delta>0$ only depending on $C_2$ and $c_2'c_3'c_4$ such that 
		\[\abs{\log \sigma_p (\rho(\wt s_1\cdots\wt s_k))-\sum_{i=1}^k\log\sigma_p(\wt s_i)}\leqslant -k\log\delta\]
		for every $\wt s_1,\cdots,\wt s_k\in \wt S$ and $1\leqslant p\leqslant n.$ Since $\psi$ is a linear function, we have
		\[\abs{\psi(\kappa(\rho(\wt s_1\cdots\wt s_k)))-\sum_{i=1}^k\psi(\kappa(\rho(\wt s_i)))}\leqslant -kC_3\log\delta, \]
		where $C_3>0$ only depends on $\psi .$ For each $\wt s\in\wt S,$ write $\wt s=sf^\varsigma,$ where $s\in S_3$ and $f\in F,\varsigma\in\lb{m,2m}.$ Then there exists $C_4>0$ only depending on $F$ and $m$ such that $|\log\sigma_p(\rho\wt s)- \log\sigma_p(\rho s)|\leqslant C_4$ for every $p.$ Hence $\psi(\kappa (\rho\wt s))\leqslant N+C_3C_4$ since $s\in S_3\sbs S_1.$
	
		Then for sufficiently large $N,$ we have
		\[\psi(\kappa(\rho(\wt s_1\cdots\wt s_k)))\leqslant kN+kC_3C_4-kC_3\log\delta \leqslant (1+\ve)kN.\]
		This implies that
		\[\psi(\lambda(\nu))=\lim_{k\to\infty}\frac{1}{k} \int \psi(\kappa(\rho(\gamma)))\dr\nu^{*k}(\gamma)\leqslant (1+\ve)N. \qedhere\]
	\end{itemize}
\end{proof}

\subsection{The variational principle of critical exponent}
This section is devoted to prove Theorem \ref{thm: variational principle}. We will also present a version for flag varieties.
In this section, we consider a Borel Anosov representation $\rho:\Gamma\to \SL_n(\RR)$  and let $\bG$ be the Zariski closure of $\rho(\Gamma).$ Recall the affinity exponent $s_{A}(\rho)$ given in \eqref{eqn: critical exponent}, which is expressed as 
\[s_{\mathrm{A}}(\rho)\defeq\sup\lb{s: \sum_{\gamma\in\Gamma}\exp(-\psi_{s}(\rho(\gamma)))=\infty } \]
where
\[\psi_s(g)\defeq\sum_{1\leqslant i\leqslant \lfloor s\rfloor }(\log\sigma_1(g)-\log\sigma_{i+1}(g))+(s-\lfloor s\rfloor)(\log\sigma_1(g)-\log\sigma_{\lfloor s\rfloor+2}(g)),\quad g\in\SL_n(\RR).\]
Note that $\psi_s(g)$ is a linear function on Cartan projection $\kappa(g).$ By abuse of the notation, we also denote $\psi_s$ to be the linear function on $\frak a$ satisfying $\psi_s(\kappa(g))=\psi_s(g).$ Then $\psi_s$ is positive with respect to $\frak a^+.$

Let $\sP_{\mr{f.s.}}(\rho)$ be the family of finitely supported probability measures on $\rho(\Gamma).$
Let $\sP_{\mr{f.s.}}^\bG(\rho)$ be the family of $\nu\in \sP_{\mr{f.s.}}(\rho)$ satisfying $G_\nu$ is Zariski dense in $\bG.$
We will also consider $\sP_{\mr{f.s.}}^{\bG^0}(\rho)$ where $\bG^0$ denotes the identity (Zariski-)component of $\bG.$
For $\nu\in\sP_{\mr{f.s}}(\rho),$ let $\mu$ be an ergodic stationary measure of $\nu$ on $\PP(\RR^d).$ 
Recall the Lyapunov dimension of $\mu$ in \eqref{eqn: Lyapunov dimension}. Then Theorem \ref{thm: variational principle} is interpreted as the following. 

\begin{prop}\label{prop: projective variation}
    Let $\rho:\Gamma\to\SL_n(\RR)$ be a Zariski dense Borel Anosov representation (that is $\bG=\SL_n(\RR)$). Then
\begin{align*}
    s_{A}(\rho)&\leqslant\sup\lb{\dim_{\mr{LY}}\mu: \mu\text{ is the unique stationary measure on $\PP(\RR^n)$ of some }\nu\in\sP_{\mr{f.s.}}^\bG(\rho)}
\end{align*}
\end{prop}

Besides, we also have a variational principle for dimensions on the flag variety. We consider
\[\psi_{F,s}(g)\defeq\inf\lb{\sum_{1\leqslant i<j\leqslant n} a_{ij}\log\frac{\sigma_i(g)}{\sigma_j(g)} :0\leqslant a_{ij}\leqslant 1,\sum_{1\leqslant i<j\leqslant n }a_{ij}=s}.\]
Note that $\psi_{F,s}(g)$ is increasing with respect to $s.$ The affinity exponent on the flag variety is given by 
%
%
\[s_{A,F}(\rho)\defeq\sup\lb{s:\sum_{\gamma\in\Gamma}\exp(-\psi_{F,s}(\rho(\gamma)))=\infty}. \]
For instance, in the case of $n=3,$ the function $\psi_{F,s}$ is given by
\[\psi_{F,s}(g)=\case{ &s\min\lb{\log\frac{\sigma_1(g)}{\sigma_2(g)},\log\frac{\sigma_2(g)}{\sigma_3(g)}},&s\leqslant 1;\\ 
&(s-1)\sb{\log\frac{\sigma_1(g)}{\sigma_2(g)}+\log\frac{\sigma_2(g)}{\sigma_3(g)}}+(2-s)\min\lb{\log\frac{\sigma_1(g)}{\sigma_2(g)},\log\frac{\sigma_2(g)}{\sigma_3(g)}},&1< s\leqslant 2;\\
&\log\frac{\sigma_1(g)}{\sigma_2(g)}+\log\frac{\sigma_2(g)}{\sigma_3(g)}+(s-2)\log\frac{\sigma_1(g)}{\sigma_3(g)} ,&2<s\leqslant 3.} \]
In general, $\psi_{F,s}(g)$ is always a minimum of finitely many linear functions on $\kappa(g)$ which is positive with respect to $\frak a^+.$ 
Specifically, $\psi_{F,s}(g)=\min\lb{\psi_{F,k}(\kappa(g)):k\in\cK_s},$ where each $\psi_{F,k}(\kappa(g))$ is of the form
\[\sum_{1\leqslant i<j\leqslant n} a_{ij}(\log\sigma_i(g)-\log\sigma_j(g))\]
satisfying $0\leqslant a_{ij}\leqslant 1,\Sigma_{1\leqslant i<j\leqslant n}a_{ij}=s$ and at most one of $a_{ij}$ is neither $0$ nor $1.$ For given $s$ and $n,$ there are at most finitely many such linear functionals. Hence $\#\cK_s$ is finite.

For $\nu\in\sP_{\mr{f.s}}(\rho),$ let $\mu$ be an ergodic stationary measure of $\nu$ on $\cF(\RR^n).$ Then the Lyapunov dimension of $\mu$ is given by 
\[\dim_{\mr{LY}}\mu=\sup\lb{\sum_{1\leqslant i<j\leqslant n} d_{ij}: 0\leqslant d_{ij}\leqslant 1,\sum_{1\leqslant i<j\leqslant n}d_{ij}(\lambda_i(\nu)-\lambda_j(\nu))=h_{\mr F}(\mu,\nu) }. \]

\begin{prop}\label{prop: flag variation}
    Let $\rho:\Gamma\to\SL_n(\RR)$ be a Borel Anosov representation. Then
    \begin{align*}
         s_{A,F}(\rho)&\leqslant \sup\lb{\dim_{\mr{LY}}\mu: \mu\text{ is an ergodic stationary measure on $\cF(\RR^n)$ of some }\nu\in\sP_{\mr{f.s.}}^{\bG^0}(\rho)}
    \end{align*}
\end{prop}
We will prove the case on flag varieties first. The case on the projective space is very similar and we will only indicate where we need to modify in the proof.
\begin{proof}[Proof of Proposition \ref{prop: flag variation}]
    
Since $\rho$ is Anosov, $\ker\rho$ is finite. Replacing $\Gamma$ by $\Gamma/\ker\rho,$ we can assume that $\rho$ is faithful. Besides, by Selberg's lemma, $\rho(\Gamma)$ is virtually torsion-free. Replacing $\Gamma$ by a finite index torsion-free subgroup, we may assume that $\Gamma$ is torsion free and the Zariski closure of $\rho(\Gamma)$ is Zariski connected. This process does not affect the value of affinity exponent and the Zariski closure of $\rho(\Gamma)$ is indeed the identity component of the original one.
Then the proposition is a direct consequence of the following lemma.
\begin{lem}
	Let $\Gamma$ be a non-elementary torsion-free hyperbolic group and $\rho:\Gamma\to\SL_n(\RR)$ a faithful Borel Anosov representation such that the Zariski closure $\bG$ of $\rho(\Gamma)$ is Zariski connected.
	Let $s>0$ such that the series $\sum_{\gamma\in\Gamma} \exp(-\psi_{F,s}(\rho(\gamma)))$ diverges. Then for every $\ve>0$ sufficiently small, there exists $\nu\in \sP_{\mr{f.s.}}^\bG(\rho)$ and an ergodic $\nu$-stationary measure $\mu$ on $\cF(\RR^n)$ satisfying 
	\[\dim_{\mr{LY}}\mu\geqslant s-\ve.\]
\end{lem}
\begin{proof}
	Recall that $\psi_{F,s}(\rho(\gamma))=\min\lb{\psi_{F,k}(\kappa(\rho(\gamma))):k\in \cK_s}$ is a minimum of finitely many linear functions. Then
	\[\sum_{k\in\cK_s}\sum_{\gamma\in\Gamma} \exp(-\psi_{F,k}(\kappa(\rho(\gamma))))\geqslant\sum_{\gamma\in\Gamma} \exp(-\psi_{F,s}(\rho(\gamma)))=\infty.\]
	This implies that there exists $k\in \cK_s$ satisfying $\sum_{\gamma\in\Gamma} \exp(-\psi_{F,k}(\rho(\gamma)))$ diverges. We fix a such $\psi_{F,k}$ in latter discussions and assume that
	\[\psi_{F,k}(g)=\sum_{1\leqslant i<j\leqslant n} a_{ij}(\log\sigma_i(g)-\log\sigma_j(g)),\quad\forall g\in\SL_n(\RR). \] 
	Now we apply Proposition \ref{prop: series variation} to $\psi_{F,k}.$ We obtain a constant $c>0$ and there exists a positive integer $N$ and $\nu\in\sP_{\mr{f.s.}}(\rho)$ satisfying
	\begin{itemize}
		\item $\supp\nu$ generates a semigroup whose Zariski closure is $\bG.$
		\item $\lambda_p(\nu)-\lambda_{p+1}(\nu)\geqslant c N$ for every $p=1,\cdots,n-1 .$
		\item $h_{\mr{RW}}(\nu)\geqslant (1-\frac{1}{2}c\ve )N$ and $\psi_{F,k}(\lambda(\nu))\leqslant (1+\frac{1}{2}c\ve )N.$
	\end{itemize}
	
    Since $\nu$ has a simple Lyapunov spectrum, there exists a $\nu$-stationary measure $\mu$ on $\cF(\RR^n)$ which corresponds to the distribution of Oseledec's splitting. Furthermore, $(\supp\mu,\mu)$ is the Poisson boundary for $(\Gamma_\nu,\nu)$ by \cite[Theorem 2.21]{furman_random_2002}, where $\Gamma_\nu$ is the group generated by $\supp\nu.$ By \cite[Theorem 2.31]{furman_random_2002}, $h_{\mr{F}}(\mu,\nu)=h_{\mr{RW}}(\nu)\geqslant (1-\frac{1}{2}c\ve )N.$ Now we estimate $\dim_{\mr{LY}}\mu.$
	
	Note that 
	\[\psi_{F,k}(\lambda(\nu))=\sum_{1\leqslant i<j\leqslant n} a_{ij}(\lambda_i(\nu)-\lambda_j(\nu))\leqslant (1+\frac{1}{2}c\ve)N.\]
	Assuming $a_{i_0 j_0}>0$ for some $1\leqslant i_0<j_0\leqslant n,$ we take 
	\[a_{ij}'=\case{&a_{ij}-\ve,&i=i_0,j=j_0;\\&a_{ij},&\text{otherwise}.}\]
	Then 
	\[\sum_{1\leqslant i<j\leqslant n} a_{ij}'(\lambda_i(\nu)-\lambda_j(\nu))\leqslant (1+\frac{1}{2}c\ve)N- \ve (\lambda_{i_0}(\nu)-\lambda_{j_0}(\nu))\leqslant (1-\frac{1}{2}c\ve)N\leqslant h_{\mr{F}}(\mu,\nu).\]
	Hence
	\[\dim_{\mr{LY}}\mu\geqslant \sum_{1\leqslant i<j\leqslant n}a_{ij}'= \sum_{1\leqslant i<j\leqslant n}a_{ij}-\ve=s-\ve. \]
	We obtain the desired conclusion.
\end{proof}
\end{proof}

\begin{proof}[Proof of Proposition \ref{prop: projective variation}]
    The only difference is to show the identity between the Furstenberg entropy and the random walk entropy. Since we assume additionally that $\rho(\Gamma)$ is Zariski dense in $\SL_n(\RR),$ the equality of two notions of entropies follows from Proposition \ref{prop:equal.entropy.geometry}. 
\end{proof}

\section{Hausdorff dimension of the Rauzy Gasket}\label{sec:rauzy}
We will verify the equality of the Haussdorff dimension of the Rauzy gasket with its affinity exponent. To simplify notations, we abbreviate $\Gamma_{\mathscr{R}}$ to $\Gamma$ in this section.
\subsection{Preliminaries and notation}
Recall that $\Delta$ is the projectivization of $\{(x,y,z): x,y,z\geq 0\}$ in $\PP(\RR^3).$ 
There is a natural bijection between $\Delta$ and the euclidean triangle $\wt{\Delta}=
\{(x,y,z)\in \R^3: x+y+z=1, x,y,z\geq 0\}$ by the projective map, which also preserves lines, hence triangles. 
The euclidean distance $d_E$ on $\wt\Delta$ from $\R^3$ is bi-Lipschitz equivalent to the projective distance $d$ coming from $\Delta$. 
Since Lipschitz constants do not affect the statements of lemmas, we identify $\Delta$ and $\wt\Delta$ in the following and we do not distinguish the euclidean metric and the projective metric.
Moreover, the area of triangles in $\Delta$ will be understood as the area of the corresponding triangle in $\wt\Delta$ in this section.

Recall that the Rauzy gasket $X\subset \P(\R^3)$ is a projective fractal set defined in the introduction. We may consider the classical coding of $X$ by infinite words as the case of IFSs. Let $\Lambda\defeq\lb{1,2,3}$ be the set of symbols. We have the following basic fact \cite[Lemma 3]{arnoux_rauzy_2013}.
\begin{lem}\label{lem: diam Rauzy}
    For every $\bi=(i_1,i_2,\cdots)\in\Lambda^\NN,$ we have $\lim_{n\to\infty} \diam A_{i_1}\cdots A_{i_n}\Delta=0.$
\end{lem}
This fact allows us to define the coding map 
\[\Phi:\Lambda^\NN\to\Delta,\quad \bi=(i_1,i_2,\cdots)\mapsto \cap_{n\in\NN} A_{i_1}\cdots A_{i_n}\Delta. \]
Then the image of $\Phi$ is exactly the Rauzy Gasket $X.$

For any $\gamma\in \Gamma$ we denote by $\Delta_\gamma$ the image $\gamma\cdot \Delta$. We also use $|\gamma|$ for the word length of $\gamma$ with respect to the standard generator set $\{A_i\}$ of $\Gamma$.
For later use, we consider the following notations. Recall that by freeness of $\Gamma$, any $\gamma\in \Gamma$ can be decomposed uniquely as the following $\gamma=A_{i_1}\cdots A_{i_{|\gamma|}}$. For any $n\leq |\gamma|$, we denote by $\gamma_n\defeq A_{i_1}\cdots A_{i_n}$.
We say that the last $n$ digits of $\gamma$ are not the same for some $n\leq |\gamma|$, if in the decomposition, $i_{|\gamma|-n+1}, \dots, i_{|\gamma|}$ are not the same. Recall $s_{\mathrm{A}}(\Gamma)$ defined in the introduction is the affinity exponent of  $\Gamma$.

We will also consider the transpose action of $\Gamma$. For any $\gamma=A_{i_1}\cdots A_{i_{|\gamma|}}\in \Gamma$, the transpose action $\gamma^t: \P(\R^3)\to \P(\R^3)$ of $\gamma$ is defined by $\gamma^t:=A_{i_{|\gamma|}}^t\cdots A_{i_1}^t$, where the transposes of $A_i$'s are 
\[A_1^t=\begin{pmatrix} 1 & 0 & 0 \\ 1 & 1 & 0\\ 1& 0& 1\end{pmatrix},\  A_2^t=\begin{pmatrix} 1 & 1 & 0 \\ 0 & 1 & 0\\ 0& 1& 1\end{pmatrix},\  A_3^t=\begin{pmatrix} 1 & 0 & 1 \\ 0 & 1 & 1\\ 0& 0& 1\end{pmatrix}.
    \]
It is not hard to check that for any $\gamma\in \Gamma$, the transpose action $\gamma^t$ preserves $\Delta$ since the entries of the matrix presentation of $\gamma^t$ 
are all non-negative. For any $n\leq |\gamma|$, we denote by $\gamma_n^t:=A_{i_{|\gamma|}}^t\cdots A_{i_{|\gamma|-n+1}}^t$. Notice that $\gamma_n^t$ is not equal to $(\gamma_n)^t$, it should be identified as $(\gamma^t)_n$. 

\color{teal}
\color{black}

\subsection{The upper bound of the Hausdorff dimension} \label{sec:upper rauzy}
The goal of this section is to show the upper bound of the Hausdorff dimension of the Rauzy gasket, that is $\dim X\leq s_{\mathrm{A}}(\Gamma)$. The following elementary lemma in linear algebra is the key observation, which plays a crucial role in estimating the upper bound of $\dim X .$

Recall that $\{e_1,e_2,e_3\}$ is the standard orthonormal basis of $\RR^3$ and $E_i=\RR e_i\in\PP(\RR^3).$



\begin{lem}\label{lem: crucial proj geo lem}
For every $n\in\N$, there exists $\epsilon_n>0$ such that for any $\gamma\in \Gamma$, if the last $n$ digits of $\gamma$ are not the same, then for $i=1,2,3$
\[\|\gamma e_i\|\geq \epsilon_n \sigma_1(\gamma). \]
\end{lem}
\begin{proof}The idea to show the lemma is to consider the transpose action of $\Gamma$. We list some basic facts on the action of $A_i^t$ on $\Delta.$ 
\begin{lem}
For every $i\ne j\in\lb{1,2,3},$ the following holds:
    \begin{enumerate}[(1)]
        \item $A_i^t$ preserves $\Delta.$
        \item $A_i^t$ preserves $\Delta',$ where $\Delta'$ is the \textbf{open} projective triangle in $\P(\R^3)$ with vertices $\begin{pmatrix} 0 \\ 1\\ 1 \end{pmatrix}, \begin{pmatrix} 1 \\ 0\\1 \end{pmatrix}, \begin{pmatrix} 1 \\ 1\\ 0 \end{pmatrix}.$
        \item $A_i^t E_j= E_j.$
        \item $A_i^t E_i=\begin{pmatrix} 1 \\ 1\\1 \end{pmatrix}\in\Delta'.$
    \end{enumerate}
\end{lem}

Combining (2)(3)(4) in the lemma above, we obtain
\begin{lem}\label{lem: claim in crucial lem proj geo}For any $i\in \{1,2,3\}$, for any $\gamma=A_{i_1}\cdots A_{i_{|\gamma|}} \in \Gamma$, if there exists some $i_j(\gamma)=i$, then $\gamma^t E_i\in \Delta'.$
\end{lem}

We are back to the proof of Lemma \ref{lem: crucial proj geo lem}. Let $\gamma\in \Gamma$ be an element such that the last $n$ digits of $\gamma$ are not the same. Due to \cref{lem:gv d v g-}, we have 
\begin{equation}\label{eqn: Rauzy proj geo}
      \|\gamma e_i\|\geq \|\gamma\|d(E_i,H_{\gamma^-}).  
\end{equation}
Recall the relation $(V_{(\gamma^t)^+})^\perp=H_{\gamma^-}.$ So it is sufficient to show the angle between $E_i$ and $V_{(\gamma^t)^+}$ is bounded away from $\pi/2$.  
Note that $\Delta$ has a special geometry property: $E_i^\perp$ is the span $\lb{E_j: j\neq i}$, which corresponds to an edge of $\Delta.$ In order to show the angle between $V_{(\gamma^t)^+}$ and $E_i$ is bounded away from $\pi/2,$ it suffices to show that $d(V_{(\gamma^t)^+},\partial \Delta)$ is lower bounded by a positive constant only depending on $n.$
To determine the position of $V_{(\gamma^t)^+},$ we use the fact that it is the attracting fixed point of $\gamma^t\gamma$ on the projective plane. Hence $V_{(\gamma^t)^+}\in \gamma^t\gamma\Delta.$
\begin{lem}
    $\gamma^t\gamma\Delta\sbs \gamma_n^t\Delta\cap \ol{\Delta'}.$
\end{lem}
\begin{proof}
    Since $A_i$ and $A_i^t$ preserve $\Delta$ for $i=1,2,3,$ we obtain $\gamma^t\gamma\Delta \sbs \gamma_n^t\Delta.$ Now we show $\gamma^t\gamma\Delta\sbs\ol{\Delta'}.$ There are two possible cases. If all of $A_1,A_2,A_3$ occur in $\gamma,$ then $\gamma^t E_i\in\Delta'$ for each $i=1,2,3$ by Lemma \ref{lem: claim in crucial lem proj geo}. Therefore $\gamma^t\gamma\Delta\sbs \gamma^t\Delta\sbs\Delta'.$

    Otherwise, there are at most two of $A_i$ occur in the $\gamma.$ Recall the assumption that the last $n$ digits of $\gamma$ are not the same. Without loss of generality, we can assume that both $A_1$ and $A_2$ occur in the last $n$ digits of $\gamma.$ Now we consider the region
    \[\nabla_z\defeq\lb{\begin{pmatrix}
        x\\y\\z
    \end{pmatrix}\in\Delta:z\leqslant x+y}.\]
    Then $A_1,A_2,A_1^t,A_2^t$ preserve $\nabla_z.$ Moreover, $A_1\Delta\sbs \nabla_z$ and $A_2\Delta\sbs\nabla_z.$ Since $A_3$ does not occur in $\gamma,$ we have $\gamma^t\gamma\Delta\sbs \nabla_z.$ Finally, we notice that the vertices of $\gamma_n^t\Delta$ satisfy $\gamma_n^tE_1,\gamma_n^tE_2\in\Delta'$ and $\gamma_n^tE_3=E_3$ by Lemma \ref{lem: claim in crucial lem proj geo}. This gives $\gamma^t\gamma\Delta\subset \gamma_n^t\Delta\cap\nabla_z=\gamma_n^t\Delta\cap\ol{\Delta'}.$
\end{proof}
\begin{figure}[!ht]
    \centering
    \includegraphics{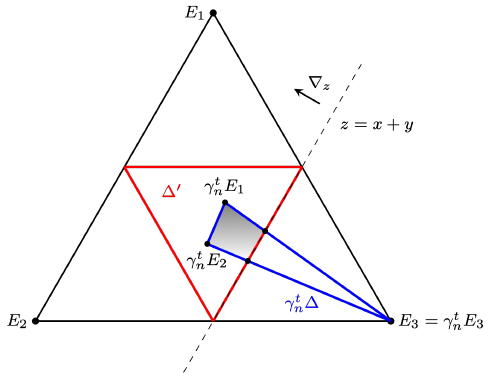}
    \caption{The pattern of $\gamma^t_n\Delta\cap \ol{\Delta'}$.}
    \label{fig: quadrilateral}
\end{figure}
We complete the proof of Lemma \ref{lem: crucial proj geo lem}. 
By an inductive argument on $n,$ we can assume that there are exactly two of $\lb{1,2,3}$ occur in $i_{|\gamma|-n+1},\cdots i_{|\gamma|}.$ Without loss of generality, we assume these two digits are $1$ and $2.$ By Lemma \ref{lem: claim in crucial lem proj geo}, the vertices of $\gamma_n^t\Delta$ satisfy
\[\gamma_n^tE_1\in\Delta',\quad \gamma_n^tE_2\in\Delta',\quad \gamma_n^tE_3=E_3. \]
Notice that set $\gamma_n^t\Delta \cap\overline{ \Delta'}$ is a closed quadrilateral (see Figure \ref{fig: quadrilateral}), which does not intersect the boundary of $\Delta$. Thus for every such $\gamma_n$, we have $ d(\partial\Delta, \gamma_n^t\Delta \cap \ol{\Delta'} )>0$. Since there are only finitely many such $\gamma_n$ for a given positive integer $n,$ there exists $d_n>0$ only depending on $n$ such that 
\[d(\partial\Delta, \gamma_n^t\Delta \cap \ol{\Delta'} )>d_n\]
for all such $\gamma_n^t$. Recalling $V_{(\gamma^t)^+}\in \gamma^t\gamma \Delta\subset \gamma_n^t\Delta \cap \Delta',$ we have $d(V_{(\gamma^t)^+},\partial \Delta)>d_n.$ 
\end{proof}

The following lemma shows some basic estimates of the diameter and the area of $\Delta_\gamma$. 

\begin{lem}\label{lem: proj geo lem}

There exists $C_2>1$ such that if the last $2$ digits of $\gamma$ are not the same, then
\begin{enumerate}[(1)]
    \item $\diam (\Delta_\gamma)\leq C_2\cdot \frac{\sigma_2(\gamma)}{\sigma_1(\gamma)}$;
    \item $\mathrm{Area}(\Delta_\gamma)\leq C_2\cdot  \sigma_1(\gamma)^{-3}$.
\end{enumerate}
\end{lem}

\begin{proof}
\begin{enumerate}[(1)]
    \item It suffices to show $d(\gamma E_i, \gamma E_j)\leqslant C_2\cdot  \frac{\sigma_2(\gamma)}{\sigma_1(\gamma)}$. 
    We have 
    \begin{eqnarray*}d(\gamma E_i, \gamma E_j)&=&\frac{\|\gamma e_i\wedge \gamma e_j\|}{\|\gamma e_i\|\cdot \|\gamma e_j\|}\leq \frac{\sigma_1(\gamma)\sigma_2(\gamma)\|e_i\wedge e_j\|}{\|\gamma e_i\|\cdot \|\gamma e_j\|}\\
    &\leq & \frac{\sigma_1(\gamma)\sigma_2(\gamma)}{\epsilon_2^2\sigma_1(\gamma)^2} \text{ (by Lemma \ref{lem: crucial proj geo lem}) 
    }\\
    &=&\epsilon_2^{-2}\cdot \frac{\sigma_2(\gamma)}{\sigma_1(\gamma)}.
    \end{eqnarray*}

\item 
We use the following elementary geometric fact: 
Let $x,y,z$ be three points in $\R^3\sm\{o\}$, then the area of the triangle formed by $x,y,z$ (which we denote by $(xyz)$) is equal to 
\[\mathrm{Area}(xyz)=\frac{\|x\wedge y\wedge z\| }{2d_E(o,(xyz))}, \]
where $d_E(o,(xyz))$ is the distance from the origin $o$ to the two plane of $(xyz)$. This is because the numerator gives the volume of the polyhedron of $oxyz$.

For $\gamma\in\Gamma$, recall that the area of $\Delta_\gamma$ is understood in the area of the corresponding subset of $\wt\Delta$ in the euclidean space.
Let $x_\gamma,y_\gamma,z_\gamma$ be the corresponding points in $\wt\Delta$ of  vertices of $\Delta_\gamma.$ Then 
\[ \mathrm{Area}(x_\gamma y_\gamma z_\gamma)=\frac{\|x_\gamma\wedge y_\gamma\wedge z_\gamma \|}{2d(o,(x_\gamma y_\gamma z_\gamma))}=\frac{\|x_\gamma\wedge y_\gamma\wedge z_\gamma \|}{2d(o,\wt\Delta)}. \]
We know that
\[x_\gamma=\gamma e_1/\omega(\gamma e_1), \]
with $\omega(v)=v_1+v_2+v_3$, similarily for $y_\gamma,z_\gamma$. Therefore, 
\[ \|x_\gamma\wedge y_\gamma\wedge z_\gamma \|=\|\gamma e_1\wedge \gamma e_2\wedge \gamma e_3 \|/\prod_{1\leq i\leq 3}\omega(\gamma e_i)  \]
We actually have $\|\gamma e_1\wedge \gamma e_2\wedge \gamma e_3 \|=\|e_1\wedge e_2\wedge e_3 \|=1$. Due to $\gamma E_i\in \Delta $, we also have 
\[\omega(\gamma e_i)\geq\|\gamma e_i\| \] 
Then by Lemma \ref{lem: crucial proj geo lem} we get the proof of (2).\qedhere
\end{enumerate}
\end{proof}

The following geometric lemma is essentially proved in  \cite[Lemma 4.1]{pollicott_upper_2021}.

\begin{lem}\label{lem: 4.1 PS}For any $\delta>0$, there exists $c_\delta>0$ such that for any $\gamma\in \Gamma$, there exists a finite open cover $\{D_i(\gamma): i=1,\cdots,k\}$ of $\Delta_\gamma$ with $\diam D_i(\gamma)\leqslant\diam \Delta_\gamma$ such that $$\sum_{i}\diam  ^{1+\delta}D_i(\gamma)\leq c_\delta \cdot \diam^{1-\delta}\Delta_\gamma\cdot \mathrm{Area}^\delta\Delta_\gamma.$$
\end{lem}
\begin{proof}
As in the proof of \cite[Lemma 4.1]{pollicott_upper_2021}, we know that every $\Delta_{\gamma}$ can be covered by $O(\diam^2(\Delta_\gamma)/\mathrm{Area}(\Delta_\gamma))$ disks $\{D_i:i=1,\cdots,k\}$ of diameter $O(\mathrm{Area}(\Delta_\gamma)/ \diam(\Delta_\gamma))$, then we get the proof. \end{proof}

We back to the proof of $\dim_{\mr H}(X)\leq s_{\mathrm{A}}(\Gamma)$. Recall $P_\Gamma(s)=\sum_{\gamma\in\Gamma}\varphi_s(\gamma)$ is the Poincar\'e series of $\Gamma$, where $\varphi_s(\gamma)$ is defined by
\[\varphi_s(\gamma)=\case{
    &\Bsb{\frac{\sigma_2}{\sigma_1}}^s(\gamma), &0<s\leq 1;\\
    &\Bsb{\frac{\sigma_2}{\sigma_1}}(\gamma)\Bsb{\frac{\sigma_3}{\sigma_1}}^{s-1}(\gamma),& 1<s\leq 2.}\]
By definition of $s_{\mathrm{A}}$, it suffices to show if $P_\Gamma(s)<\infty$ then for any $\epsilon>0$, $\dim X\leq s+\epsilon$.

\begin{defi}\label{def: nice word}
    For $x\in X$, we say $x$ is \textit{nice} if every $(i_1(x),i_2(x),\cdots)\in\Phi^{-1}(x)\in \Lambda^\N$ is not ending by a single element in $\Lambda=\{1,2,3\}$.
\end{defi}
We remark that if $x$ is nice, then $x$ is uniquely coding. This is because the only possibility of $\Phi(\bi)=\Phi(\bi')$ is 
\[\bi= (w,j_1,j_2,j_2,\cdots)\text{ and }\bi'=(w,j_2,j_1,j_1,\cdots),\]
where $w\in\Lambda^N$ for some $N\in\NN$ and $j_1\ne j_2\in\Lambda.$

Then the Rauzy Gasket $X$ can be decomposed as the set of nice points, which we denote it by $X_{\mathrm{nice}}$, and a countable set. It suffices to show the Hausdorff outer measure $H^{s+\epsilon}(X_{\mathrm{nice}})$ is $0$.


We consider
\[\Gamma_m\defeq\lb{\gamma\in\Gamma:\text{the last $2$ digits of $\gamma$ are not the same and }\diam \Delta_\gamma\leqslant 1/m}.\]
Now we construct two families of covers $\cU_{m}$ and $\cU_m'$ for $m\in\N$ as 
\[\cU_m\defeq\lb{D_i(\gamma):\gamma\in\Gamma_m},\quad\cU_m'\defeq\lb{\Delta_\gamma:\gamma\in\Gamma_m},\]
where $\{D_i(\gamma)\}$ is the finite open cover of $\Delta_\gamma$ we obtained from Lemma \ref{lem: 4.1 PS}.
Let
\[Y\defeq\bigcap_{m=1}^\infty \bigcup_{U\in\cU_{m}} U\quad\text{and}\quad Y'\defeq\bigcap_{m=1}^\infty \bigcup_{U\in \cU'_{m}} U.\]
Then the sequence of covers $\cU_{m}$ (resp. $\cU_m'$) is a family of Vitali covers of the set $Y$ (resp. $Y'$).\footnote{Recall that a Vitali cover $\mathcal V$ of a set $E$ is a family of sets so that, for every $\delta>0$ and every $x \in E$, there is some $U$ in the family $\cV$ with $\diam U <\delta$ and $x\in U$.} Notice that $Y$ contains $Y'$ by the construction of $D_i(\gamma)$'s.

\begin{lem} \label{lem:Xnice}
$Y\supset  X_{\mathrm{nice}}$.
\end{lem}
\begin{proof}We claim that for any $\bi \in \{1,2,3\}^\N$ which is not ending by a single element, the point $x=\Phi(\bi)$ is contained in $Y$. Since $\bi=(i_1(x),\cdots )$ is not ending by a single element in $\{1,2,3\}$, there exists infinitely many $\ell$ such that $i_{\ell-1}(x), i_\ell(x)$ are not the same (otherwise it will be ended by only one element). Collect all such $
\ell$ and consider all the elements $\gamma=A_{i_1}
\cdots A_{i_{\ell}}$. Then we get there are infinitely many $\gamma\in \Gamma$ such that the last two digits of $\gamma$ are not the same and $x\in \Delta_\gamma$. By \cref{lem: diam Rauzy}, $\diam \Delta_\gamma\to 0$ as $\ell$ tending to infinity. Therefore $x\in Y'\subset Y$.
\end{proof}
As a consequence, 
it suffices to show the Hausdorff outer measure $$H^{s+\epsilon}(Y)=\lim_{\delta\to 0}H^{s+\epsilon}_\delta(Y)=0.$$ 
Recall that $\diam U\leqslant 1/m$ for every $U\in\cU_m.$
Then for $s\geq 1$, we have
\begin{eqnarray*}
H^{s+\epsilon}(Y)&\leq& \limsup_{\delta\to 0}H^{s+\epsilon}_\delta(Y)~\leqslant ~\limsup_{m\to \infty} 
\sum_{U\in \cU_{m}}\diam (U)^{s+\epsilon}\\
&= & \limsup_{m\to \infty}\sum_{\gamma\in\Gamma_m} \sum_i\diam (D_i(\gamma))^{s+\epsilon}\\
&\leq & \limsup_{m\to \infty}  c_{s+\epsilon-1}\sum_{\gamma\in\Gamma_m}\diam^{2-s-\epsilon}\Delta_\gamma\cdot \mathrm{Area}^{s+\epsilon-1}\Delta_\gamma \qquad (\text{by Lemma \ref{lem: 4.1 PS}}) \\
&\leq & \limsup_{m\to \infty}c_{s+\epsilon-1}C_2\sum_{\gamma\in\Gamma_m} \varphi_{s+\epsilon}(\gamma)\text{ 
\quad(by Lemma \ref{lem: proj geo lem} and the definition of $\varphi_{s}$}).
\end{eqnarray*}

Since $\sum \varphi_{s+\epsilon}(\gamma)<\infty$, the right-hand side of the last inequality is $0$ when $m\to \infty$, which is exactly what we want. For the case $s<1$ (we actually know $s\sim 1.72$ by numerical test), we do not need to use Lemma \ref{lem: 4.1 PS} and consider $Y'$ instead of $Y$ by the same argument. We still get the proof. 

\subsection{The lower bound of the Hausdorff dimension}\label{sec.low rauzy}

In this section, we will show that $\dim X\geqslant s_{\mathrm{A}}(\Gamma).$ Note that $X$ contains the boundary of $\Delta,$ which has the Hausdorff dimension equal to $1.$ Therefore, we already have an a priori estimate that $\dim X \geq 1$, which allows us to assume, without loss of generality, that $s_{\mathrm{A}}(\Gamma) > 1$.
In other words, we assume the existence of a value $s > 1$ for which the series $\sum_{\gamma}\varphi_s(\gamma)$ diverges. This technical assumption simplifies subsequent discussions.
Recall that for every exact dimensional probability measure $\mu$ supported on $X,$ we have $\dim\mu\leqslant \dim X.$ The inequality $\dim X\geqslant s_{\mathrm{A}}(\Gamma)$ is a direct consequence of the following variational principle of the affinity exponent.

\begin{lem}\label{lem: Rauzy variation}
    Let $s>1$ such that the series $\sum_{\gamma\in\Gamma}\varphi_s(\gamma)$ diverges, then for every $\ve>0,$ there exists a finitely supported measure $\nu$ supported on $\Gamma$ and a $\nu$-stationary measure $\mu$ supported on $X$ satisfying $\dim \mu\geq s-\ve$.
\end{lem}

Our idea is a stopping time argument which is partially inspired by the study of the variational principle of iterated function systems \cite{feng2009dimension} and self-affine measures \cite{morris_variational_2023}. Specifically, our strategy is to combine our dimension formula of stationary measures with a modification of the proof of Theorem 3.1 in \cite{morris_variational_2023}. The main difference of our proof from \cite{morris_variational_2023} is that due to the presence of unipotent elements, the word lengths of elements may lose control. So instead of considering the word length, we only consider $\kappa(\gamma).$ This is enough to estimate the Lyapunov exponents. Then we use a combinatorial argument to make elements generating a free semigroup. This generates a large enough random walk entropy. Therefore, we can find a stationary measure with a large dimension approximating the affinity exponent. A similar argument to address the issue of words with uncontrolled length also appears in \cite{HJX}.


In order to apply the dimension formula (Theorem \ref{thm:lyapunov}), we need to show the Zariski density of $\Gamma$ at first.
\begin{lem}
The semigroup $\Gamma$ is Zariski dense in $\SL_3(\R)$.
\end{lem}
\begin{proof}
	Let $\bH$ be the Zariski closure of semigroup $\Gamma$ in $\SL_3(\R)$, which is a real algebraic group.
    Let $\frak h$ be the Lie algebra of the connected component of $\bH$.
	By \cite[Page 106, Remark]{borel_linear_1991}, the fact that $A_1$ is unipotent implies that the one-parameter unipotent group $\left\{s\in\R,\  \begin{pmatrix}
	1 & s & s \\ 0 & 1 & 0\\ 0 & 0& 1
	\end{pmatrix}\right\}$ is in $\bH$ and  $X_1=\begin{pmatrix}
	0 & 1 & 1\\ 0 & 0 &0 \\ 0 & 0 & 0
	\end{pmatrix}$ is in $\frak h$. Similarly, the nilpotent elements $X_2=\begin{pmatrix}
	0 & 0 & 0\\ 1 & 0 &1 \\ 0 & 0 & 0
	\end{pmatrix}$, $X_3=\begin{pmatrix}
	0 & 0 & 0\\ 0 & 0 &0 \\ 1 & 1 & 0
	\end{pmatrix}$ are also in $\frak h$. Then we can play with these elements in the Lie algebra $\frak h$ and prove that they generate all $\frak{sl}_3$. We have
	$Y_3=[X_1,X_2]=\begin{pmatrix}
	1 & 0 & 1\\ 0 & -1 &-1 \\ 0 & 0 & 0
	\end{pmatrix}$, $Y_1=[X_2,X_3]=\begin{pmatrix}
	0 & 0 & 0\\ 1 & 1 &0 \\ -1 & 0 & -1
	\end{pmatrix}$, $Y_2=[X_3,X_1]=\begin{pmatrix}
	-1 & -1 & 0\\ 0 & 0 &0 \\ 0 & 1 & 1
	\end{pmatrix}$ also in $\frak h$. Then $[Y_1,X_1]+Y_2+Y_3+2X_1=\begin{pmatrix}
	0 & 0 & 4\\ 0 & 0 &0 \\ 0 & 0 & 0
	\end{pmatrix} $, $[Y_2,X_2]+Y_3+Y_1+2X_2=\begin{pmatrix}
	0 & 0 & 0\\ 4 & 0 &0 \\ 0 & 0 & 0
	\end{pmatrix} $ also in $\frak h$. From these elements, we can obtain the whole $\frak{sl}(3,\R)$ and then $\bH$ must be the whole group $\SL_3(\R)$.
\end{proof}

Recall that for $s>1,$ the linear function $\psi_s$ on $\frak a$ is given by 
\[\psi_s(\lambda)=(\lambda_1-\lambda_2)+(s-1)(\lambda_1-\lambda_3),\quad\forall \lambda\in\frak a.\]
Now we construct a good set coming from the divergent series $\sum_{\gamma}\vp_s(\gamma)=\sum_\gamma \exp(-\psi_s(\kappa(\gamma))).$
\begin{defi}
    For $\beta>0,x\in\frak a^+$ and $n\in\NN,$ a subset $S\sbs \Gamma$ is called \textit{$(n,\beta,x)$-approximate} if $\# S\geqslant e^{(1-\beta)n}$ and for every $\gamma\in S,$ $\|\frac{1}{n} \kappa(\gamma)-x\|\leqslant \beta.$
\end{defi}

\begin{lem}\label{lem:s nk}
For any $\beta>0$, there exists $x\in\frak a^+$ satisfying $|\psi_s(x)-1|\leq \beta$ and infinitely many $n\in\NN$ such that there exists an $(n,\beta,x)$-approximate subset $S\sbs \Gamma.$
\end{lem}

\begin{proof}
    Note that $\psi_s(\kappa(\gamma))\geqslant (s-1)\log\sigma_1(\gamma)=(s-1)\log\nm \gamma.$ Hence for every $t>0,$ the set $\lb{\gamma\in\Gamma:\psi_s(\kappa(\gamma))<t}$ is finite by the discreteness of $\Gamma.$ 
    Combining with the hypothesis that the series $\sum_{\gamma\in\Gamma} \exp(-\psi_s(\kappa(\gamma)))$ diverges, there are infinitely many $n$ such that 
    \[\lb{\gamma: \psi_s(\kappa(\gamma))\in [(1-\beta/10)n,(1+\beta/10)n]}\]
    contains at least $e^{(1-\beta/10)n}$ elements.

    Since we have $\psi_s(\kappa(\gamma))\geqslant (s-1)\log\sigma_1(\gamma)$ and $\psi_s(\kappa(\gamma))\geqslant -(s-1)\log\sigma_3(\gamma),$ the set $\psi_s^{-1}([1-\beta/10,1+\beta/10])\cap \frak{a}^+$ is compact. Therefore we can cover it by finitely many balls of radius $\beta/100.$ By a pigeonhole principle, there exists a center of some ball, say $x\in\frak a^+,$ such that for infinitely many $n,$
    \[\lb{\gamma:|\frac{\psi_s(\kappa(\gamma))}{n}-1|\leqslant \frac{\beta}{10},~\|\frac{\kappa(\gamma)}{n}-x\|\leqslant \frac{\beta}{100} }\]
    contains at least $ e^{(1-\beta)n}$ elements. Using the fact that $|\psi_s(\kappa(\gamma))|\leqslant 10\|\kappa(\gamma)\|,$ we obtain $|\psi_s(x)-1|\leqslant\beta,$ which is exactly what we want.
\end{proof}

We will modify an $(n,\beta,x)$-approximate subset to a set which freely generates a free semigroup. Since $\Gamma$ is itself a free semigroup, a subset of $\Gamma$ can freely generate a free semigroup by avoiding the ``prefix relations''. We consider the following concepts.

\begin{defi}
\begin{enumerate}[(1)]
    \item An element $j_1\in \Gamma$ is called \textit{starting with} $j_2\in \Gamma$ if there is $j_3\in \Gamma\sm\{\id\}$ such that $j_1=j_2j_3$. We also say $j_2$ is a \textit{prefix} of $j_1$ if $j_1$ is starting with $j_2.$ 
    \item An element $j_1\in \Gamma$ is called \textit{ending with} $j_2\in \Gamma\sm\{\id\}$ if there is $j_3\in \Gamma$ such that $j_1=j_3j_2$. 
    \item An element $j$ is called \textit{minimal} in a subset $S$ of $\Gamma$ if there is no element $j'\in S$ such that $j$ is starting with $j'$.
\end{enumerate}
\end{defi}
Within a set $S$, a minimal element of $S$ is never a prefix of another minimal element. So the set of minimal elements $S_{\min}$ of $S$ will freely generate a free semigroup. Moreover, the subset $S_{\min}^{*\ell}\subset \Gamma$ satisfies that there is no pair of elements such that one is a prefix of the other. In the following lemma, using hyperbolicity and discreteness, we get a lower bound of minimal elements in a set with approximately the same sizes of Cartan projections.

\begin{lem}\label{lem:multiplicity j}
There exists $C>0$ such that the following hold.
For every $\beta>0,$ $n\in \NN$ and $x\in\frak a^+.$ For any element $j$ in the set
\[\calJ=\{\gamma\in\Gamma:\ \|\frac{\kappa(\gamma)}{n}-x\|\leq \beta,~~\gamma \text{ is ending with }A_1A_2A_3 \},\]
there are at most $e^{C\beta n}$ elements in $\cJ$ which is starting with $j.$
\end{lem}

\begin{proof}
For every $V\in\Delta$, let $v=\sum_{i}v_ie_i\in V$ be a unit vector with $v_1,v_2,v_3\geqslant 0.$
Since all the entries of $\gamma e_i$ are non-negative, we have
\[ \|\gamma v\|\geq \max_i\{v_i\|\gamma e_i \|\}\geq \frac{1}{2}\min_i\{\|\gamma e_i\|\}. \]
Take $n=2$ in \cref{lem: crucial proj geo lem}, for any element $\gamma$ with last two digits different, we have  
$\|\gamma e_i\|\geq \epsilon_2\sigma_1(\gamma)$, where $\epsilon_2$ is defined in \cref{lem: crucial proj geo lem}. Hence 
\[\|\gamma v\|\geq \frac{\epsilon_2}{2}\sigma_1(\gamma).\]
Then for any $j, jj''\in \calJ$ with $j''\in\Gamma\sm\{\id\}$ and any unit $v\in V\in\Delta$, since $j''V\in\Delta$ and $j,j''$ end with $A_1A_2A_3$, we obtain
\begin{equation}\label{eqn: j,j'' sigma1}
 \sigma_1(jj'')\geq \|jj''v\|=\frac{\|j(j''v)\|}{\|j''v\|}\frac{\|j''v\|}{\|v\|}\geq \frac{\epsilon_2^2}{4}\sigma_1(j)\sigma_1(j'').     
\end{equation}

Since 
$\|\frac{1}{n}\kappa(jj'')-x\|\leq \beta$ and  $\|\frac{1}{n}\kappa(j)-x\|\leq \beta $, we have 
\begin{equation}\label{eqn: j, j'' sigma upper}
 \log \sigma_1(jj'')\leq 2\beta n+\log\sigma_1(j).   \end{equation}
Combining \eqref{eqn: j,j'' sigma1},\eqref{eqn: j, j'' sigma upper} we get 
\[\sigma_1(j'')\leq 4e^{2\beta n}/\epsilon_2^2. \]
Since the semigroup $\Gamma$ is discrete in $\SL_3(\R)$, up to a constant the number
\[\#\{\gamma\in\Gamma,\ \|\gamma\|\leq t \}  \]
is bounded by the volume of the set $\{g\in \SL_3(\R),\ \|g\|\leq t \} $ , which grows at most polynomially on $t$. Then the possible number of $j''$ is bounded by $q^{C\beta n}$ for some constant $C>0.$ 
\end{proof}

In order to estimate the Lyapunov exponents of the constructing measure, we need to estimate the Cartan projection of products. Let us recall some notions.
\begin{defi}
\begin{enumerate}[(1)]
    \item For an element $g\in \SL_3(\R)$, we call it \textit{$(r,\epsilon)$-loxodromic} for $r,\epsilon>0$, if $\sigma_i(g)/\sigma_{i+1}(g)\geq 1/\epsilon$ and 
\[d(V_{g}^+,H^-_{g} )>r,\ d(V_{g,\wedge^2\R^3}^+,H^-_{g,\wedge^2\R^3} )>r, \]
where $\wedge^2\R^3$ is the wedge representation of $\SL_3(\R)$ and $ V^+_{g,\wedge^2\R^3}=\wt k_g(E_1\wedge E_2)$ and $H^-_{g,\wedge^2\R^3}=k_g^{-1}(E_2\wedge E_3\oplus E_3\wedge E_1)$ are the corresponding attracting point and repelling hyperplane in the projective space $\P(\wedge^2\R^3),$ where $g=\wt k_ga_gk_g\in KA^+K$ is the Cartan decomposition.

\item Let $F$ be a subset in $\SL_3(\R)$. We call $F$ a \textit{$(r,\epsilon)$-Schottky family} if every element in $F$ is $(r,\epsilon)$-loxodromic and for any pair $(g,h)$ in $F$, we have
\[d(V_{g}^+,H^-_{h} )>6r,\ d(V_{g,\wedge^2\R^3}^+,H^-_{h,\wedge^2\R^3} )>6r.\]

\item A set $F$ of $\SL_3(\R)$ is called \textit{$\eta$-narrow} if for $g$ in $F$, the attracting points $V_g^+$ (resp. $V_{g,\wedge^2\R^3}^+$) are within $\eta$-distance of one another and the repelling hyperplanes $H_g^-$ (resp. $H_{g,\wedge^2\R^3}^-$) are within $\eta$ Hausdorff distance of one another. 

\item  A set $F$ is \textit{$\eta$-narrow around $h$} if for $g$ in $F$, the attracting points $V_g^+$ (resp. $V_{g,\wedge^2\R^3}^+$) are within $\eta$-distance to $V_h^+$(resp. $V_{h,\wedge^2\R^3}^+$) and the repelling hyperplanes $H_g^-$ (resp. $H_{g,\wedge^2\R^3}^-$) are within $\eta$ Hausdorff distance to $H_h^-$ (resp. $H_{h,\wedge^2\R^3}^-$). 
\end{enumerate}
\end{defi}
These definitions and properties are originally due to Benoist \cite{benoist_proprietes_1997}. We borrow them from \cite[Corollary 2.16, 2.17]{morris_variational_2023}
\begin{lem}\label{lem:cartan sum}
\begin{enumerate}[(1)]
    \item For $r>4\epsilon>0$, if $F$ is a $(r,\epsilon)$-Schottky family, then the semigroup generated by $F$ is a $(r/2,2\epsilon)$-Schottky family.

    \item Let $E$ be a $\eta$-narrow collection of $(r,\epsilon)$-loxodromic elements with $r> 4\max\{\epsilon,\eta \}$. Then, $E$ is a $(r/4,\epsilon)$-Schottky family.
    
    \item If $F$ is a $(r,\epsilon)$-Schottky family, then there exists $C_r>0$ only depending on $r$ such that for any $g_1,\cdots, g_\ell$ in $F$, we have
\[\|\kappa(g_1\cdots g_\ell)-\sum_{1\leq i\leq \ell}\kappa(g_i) \|\leq \ell\cdot C_r.  \]
\end{enumerate}
\end{lem}


Now we state the main construction, which gives a good set to support a desired random walk.

\begin{lem}\label{lem:cal J}
For every $\beta>0.$
There exists $N\in \N$, $x\in\frak a^+, \cal J\subset \Gamma,$ 
such that 
\begin{enumerate}[(1)]
\item $|\psi_s(x)-1|\leq \beta$.
\item The semigroup generated by $\cal J$ is Zariski dense.
\item For every $k\in\N$ and $j\in \calJ^{*k}:=\{j_1 \cdots j_k: j_i\in \calJ \}$, $\|\frac{1}{kN}\kappa(j)-x \|\leq 10\beta$. 
\item The set $\cal J$ contains a subset $\cal J_1$ satisfies $\#\calJ_1\geq e^{(1-10C\beta )N}$ and no element in $\calJ_1$ is a prefix of another one, where $C>0$ is the constant given by Lemma \ref{lem:multiplicity j}.
\end{enumerate}
\end{lem}
\begin{proof}

Fix an $x\in \frak a^+$ be given by Lemma \ref{lem:s nk}.
Let $S$ be a $(n,\beta,x)$-approximate subset for some sufficiently large $n.$
Now we construct the set $\cal J$ by modifying $S$.

\paragraph{Step 1.}
For every $\gamma\in S$, we add $A_1A_2A_3$ at the end and denote the new set by $W_2$. Since $\|\kappa(\gamma A_1A_2A_3)-\kappa(\gamma)\|\leq \|\kappa(A_1A_2A_3)\|$, the set $W_2$ is $(n,2\beta,x)$-approximate for $n$ large enough.
\paragraph{Step 2.}
We apply a theorem of Abels-Margulis-Soifer \cite{abels1995semigroups}, see also \cite[Theorem 3.2]{morris_variational_2023}. 
\begin{thm}[Abels-Margulis-Soifer]\label{thm: AMS}
 Let $G$ be a Zariski-connected real reductive group and $\Gamma$ be a Zariski-dense subsemigroup. Then there exists $0 < r = r(\Gamma)$ such that for all $0 < \epsilon \leq  r$, there exists a finite subset $F = F(r,\epsilon,\Gamma)\subset \Gamma$ with the property that for every $g \in  G$, there exists $f\in F$ such that $fg$ is $(r, \epsilon)$-loxodromic in $G$.
\end{thm}
We fix $r_0,\epsilon_0>0$ sufficiently small comparing to $\beta$ with $r_0>100\epsilon_0.$
By \cref{thm: AMS}, we could find a finite subset $F_1=F
(r_0,\epsilon_0,\Gamma)$ of $\Gamma.$
Therefore for every element $\gamma\in W_2$, there exists $f\in F_1$ such that $f\gamma $ is $(r_0,\epsilon_0)$-loxodromic. 
By the pigeonhole principle, we can find an $f\in F_1$ such that for at least $(\#F_1)^{-1}$ proportion of $\gamma$ in $W_2$, the product $f\gamma$ is $(r_0,\epsilon_0)$-loxodromic. Fix this $f\in F_1$ and let
\[W_3=\lb{f\gamma:\gamma\in W_2, f\gamma \text{ is $(r_0,\epsilon_0)$-loxodromic }}.\]
Then $W_3$ is $(n,3\beta,x)$-approximate assuming $n$ large enough.
\paragraph{Step 3.}
By compactness, we can cover $\prod_{1\leq i\leq 2}\P(V_i)\times\P(V_i^*)$ with 
 $O(\epsilon_0^{-8})$ balls of radius $\epsilon_0$, where $V_1=\R^3$ and $V_2=\wedge^2\R^3$. By the pigeonhole principle, there exists a subset $W_4\sbs W_3$, such that $\#W_4\gg\epsilon_0^{8}\cdot\#W_3$ and $W_4$ is an $\epsilon_0$-narrow set of $(r_0,\epsilon_0)$-loxodromic elements. For $n$ sufficiently large comparing to $\epsilon_0,$ $W_4$ is $(n,4\beta,x)$-approximate. 

\paragraph{Step 4.}
Before making the next modification, we recall the following lemma from \cite{benoist_proprietes_1997} and \cite{morris_variational_2023}.
\begin{lem}[{\cite[Lemma 3.4]{morris_variational_2023}}]
    There exists $r_1>0$ depending only on $\Gamma$ such that the following hold. For every loxodromic element $g\in G$ and $\epsilon>0,\eta>0$ there exists a Zariski dense $(r_1,\epsilon)$-Schottky subgroup of $\Gamma$ which is $\eta$-narrow around $g.$
\end{lem}

Since $r_1$ is determined by $\Gamma,$ we can assume at first that $r_0<r_1.$
Now we fix an element $g\in W_4.$
Then we can find a Zariski dense $(r_0,\epsilon_0)$-Schottky subgroup $\Gamma_1$ of $\Gamma$ which is $\epsilon_0$-narrow around $g.$ 
By the proof of $k=0$ in Lemma \ref{lem: independent elements} (see also \cite[Lemma 3.6]{morris_variational_2023}), we can find a finite subset $\{\theta_i: i=1,\cdots, p \}\subset \Gamma_1,$ which generates Zariski dense sub-semigroup.
Let $W'=W_4\cup\lb{\theta_i:i=1,\cdots,p},$ which consists of $(r_0,\epsilon_0)$-loxodromic elements. Moreover, every element in $W'$ is $\epsilon_0$-narrow around $g.$ Hence $W'$ is $2\epsilon_0$-narrow. Therefore, $W'$ is a $(r_0/4,\epsilon_0)$-Schottky family and the semigroup it generates is a $(r_0/8,2\epsilon_0)$-Shottky family, by Lemma \ref{lem:cartan sum}.


Take an $m\in\NN$ large enough depending on $\theta_i$, $\beta$ and $\epsilon_0$. Let
\[W_5\defeq W_4^{*m}\cup \{\theta_i g^m,\ i=1,\cdots, p \} \]
Now we verify that the set $\cal J\defeq W_5$ satisfies the condition for $N=nm.$
\begin{enumerate}[(1)]
    \item This is because $x$ is given by Lemma \ref{lem:s nk}.
    \item Let $\bH$ be the Zariski closure of the semigroup generated by $W_5,$ which is an algebraic subgroup of $\SL_3(\RR).$ Note that $g^m\in W_4^{*m}\sbs W_5$ and $\theta_ig^m\in W_5$ for every $i.$ We obtain $\theta_i\in\bH.$ Since $\lb{\theta_i:i=1,\cdots,p}$ generates a Zariski dense subgroup, we have $\bH=\SL_3(\RR).$

\item Recall that the semigroup generated by $W'$ is a $(r,\epsilon)$-Schottky family, where $r=r_0/8$ and $\epsilon=2\epsilon_0.$ 
Also recall that $W_4$ is $(n,4\beta,x)$-approximate. 
By \cref{lem:cartan sum}, for every $h\in W_4^{*m},$ we have
\[\|\frac{\kappa(h)}{N}-x \|=\|\frac{\kappa(h)}{nm}-x\|\leq \frac{C_r}{n}+4\beta.  \]
For $h\in \{\theta_i g^m, i=1,\dots,p\}$, we have
\[\|\frac{\kappa(h)}{N}-x \|=\|\frac{\kappa(h)}{nm}-x\|\leq \frac{C_r}{n}+4\beta+\frac{\nm{\kappa(\theta_i)}}{nm}.\]
By taking $m$ large enough depending on $\theta_i,\beta$ and then taking $n$ large enough, we assume that $C_r/n+\nm{\kappa(\theta_i)}/(nm)< \beta.$ Therefore, the set $\cJ=W_5$ is $(N,5\beta,x)$-approximate.
Since the semigroup generated by $W_5$ is an $(r,\epsilon)$-Schottky family, for every $j\in\cJ^{*k},$ we have
\[\|\frac{\kappa(j)}{kN}-x\|\leqslant\frac{C_r}{N}+5\beta\leqslant 10\beta.\]
\item
We take minimal elements $(W_4)_{\min}$ in $W_4$. Let $\cJ_1=((W_4)_{\min})^{*m}.$ Then there is no element in $\cJ_1$ which is a prefix of another element. Note that $W_4$ is $(n,4\beta,x)$-approximate and every element in $W_4$ is ending with $A_1A_2A_3.$ By \cref{lem:multiplicity j}, we have 
\[\#\cJ_1=(\#(W_4)_{\min})^m\geq (\#W_4/e^{4C\beta n})^m\geq e^{(1-4\beta -4C\beta )nm}\geqslant e^{(1-10C\beta)N}  .\qedhere \]
\end{enumerate}
\end{proof}

Finally, we will construct the random walk and estimate the dimension of the stationary measure. This part is to complete the proof of Lemma \ref{lem: Rauzy variation}.

Let $\cJ,\cJ_1$ be the sets given by Lemma \ref{lem:cal J}. 
We take $\nu=(1-\beta)\nu_1+\beta\nu_2,$ where $\nu_1$ is the uniform measure on $\calJ_1$ and $\nu_2$ is the uniform measure on $\calJ\sm\calJ_1$. 
Then by \cref{lem:cal J} the support of $\nu$ generates a Zariski dense subgroup in $\SL_3(\R)$ and the associated Lyapunov vector is close to $Nx.$ Hence, by \cref{lem:cal J}(3),
\begin{equation}
    |\frac{\psi_s(\lambda(\nu))}{N}-1|\leq|\psi_s(x)-1|+5\cdot\|\frac{\lambda(\nu)}{N}-x\|\leqslant 100\beta.
\end{equation}
Now we should estimate the random walk entropy of $\nu.$ Firstly, note that $\cJ_1$ freely generates a free semigroup, we have
\[h_{\mr{RW}}(\nu_1)=H(\nu_1)\geq (1-10C\beta)N.\]
Moreover, since the support of $\nu_1$ satisfies that no element is a prefix of another one, by certain ``continuity" property of the random walk entropy and freeness of $\Gamma$, we have
\begin{lem}Let  $\nu,\nu_1,\nu_2$ be probability measures which supported on $\Gamma$ such that $\nu=(1-\beta)\nu_1+\beta\nu_2$. If the support of $\nu_1$ is a minimal set (i.e. no element is a prefix of another), then
\begin{equation}
h_{\mr{RW}}(\nu)\geq (1-\beta)h_{\mr{RW}}(\nu_1). 
\end{equation}
\end{lem}
\begin{proof}
By the concavity of the entropy function $H$, 
\[H(\nu_1^{*\ell_1}*\nu_2^{*j_1}*\nu_1^{*\ell_2}*\cdots* \nu_2^{*j_k})\geq \int H(\nu_1^{*\ell_1}*\delta_{g_1}*\nu_1^{*\ell_2}*\cdots* \delta_{g_k})\dd \nu_2^{*j_1}(g_1)\cdots \dd\nu_2^{*j_k}(g_k).\]
We claim that $H(\nu_1^{*\ell_1}*\delta_{g_1}*\nu_1^{*\ell_2}*\cdots* \delta_{g_k})=H(\nu_1^{*(\ell_1+\cdots+\ell_k)})$. It suffices to show that all the elements in the support of LHS are distinct. Otherwise suppose that two elements $(h_1,\cdots,h_k)$ and $(h_1',\cdots, h_k')$ in $\supp\nu_1^{*\ell_1}\times \cdots \times \supp\nu_1^{*\ell_k} $ satisfy $h_1g_1\cdots h_kg_k=h_1'g_1\cdots h_k'g_k $. Since
no element in $\supp\nu_1$ is a prefix of another one, so does $ \supp\nu_1^{*\ell_1}$. Therefore by freeness of $\Gamma$, we obtain that $h_1=h_1'$. By induction we get $h_2=h_2',\cdots,h_k=h_k'$. Therefore
\[H(\nu_1^{*\ell_1}*\nu_2^{*j_1}*\nu_1^{*\ell_2}*\cdots* \nu_2^{*j_k})\geq \int H(\nu_1^{*(\ell_1+\cdots+\ell_k)})\dd \nu_2^{*j_1}(g_1) \cdots d\nu_2^{*j_k}(g_k)=H(\nu_1^{*(\ell_1+\cdots+\ell_k)}).\]

Combining with the concavity of $H$, we get 
\begin{eqnarray*}
    h_{\mr{RW}}((1-\beta)\nu_1+\beta\nu_2)&=&\lim_{\ell\rightarrow\infty}\frac{1}{\ell}H(((1-\beta)\nu_1+\beta\nu_2)^{*\ell})\\
    &\geq& \lim_{\ell\rightarrow\infty}\frac{1}{\ell}\sum_j\binom{l}{j}(1-\beta)^{\ell-j}\beta^jH(\nu_1^{*(\ell-j)})\\
    &=&\lim_{\ell\rightarrow\infty}\frac{1}{\ell}\sum_j\binom{\ell}{j}(1-\beta)^{\ell-j}\beta^j(\ell-j)h_{\mr{RW}}(\nu_1)\\
    &\geq&\lim_{\ell\rightarrow\infty}(1-\beta)\sum_j\binom{\ell-1}{j}(1-\beta)^{\ell-j-1}\beta^jh_{\mr{RW}}(\nu_1)\\
    &=&(1-\beta)h_{\mr{RW}}(\nu_1).
\end{eqnarray*}
\end{proof}
As a consequence $h_{\mr{RW}}(\nu)\geq (1-\beta)h_{\mr{RW}}(\nu_1)\geq (1-\beta)(1-10 C\beta)N$. Moreover, the group generated by $\supp \nu$ is Zariski dense in $\SL_3(\R)$ and satisfies exponential separation property (actually discreteness). 
Let $\mu$ be the unique $\nu$-stationary measure.
By the dimension formula, i.e. Theorem \ref{thm:lyapunov}, we have $\dim\mu=\dim_{\mr{LY}}\mu.$

We admit the fact that $h_{\mr{F}}(\mu,\nu)=h_{\mr{RW}}(\nu)$ and postpone the proof to the end of this section. We give an estimate on $\dim_{\mr{LY}}\mu.$ Since $\psi_s(\lambda(\nu))\geqslant (1-100\beta)N,$ we obtain
\[s(\lambda_1(\nu)-\lambda_3(\nu))\geqslant \psi_s(\lambda(\nu))\geqslant (1-100\beta)N.\]
Then $\lambda_1(\nu)-\lambda_3(\nu)\geqslant \frac{1}{2s}N$ assuming $\beta$ small enough. For a given $0<\ve<s-1,$ we have
\[\psi_{s-\ve}(\lambda(\nu))=\psi_s(\lambda(\nu))-\ve (\lambda_1(\nu)-\lambda_3(\nu))\leq (1+100\beta )N-\frac{\ve}{2s}N.\]
Now we take $\beta>0$ sufficiently small comparing to $\ve,$ we obtain 
\[\psi_{s-\ve}(\lambda(\nu)) \leq (1+100\beta )N-\frac{\ve}{2s}N\leqslant (1-\beta)(1-10C\beta)N\leqslant h_{\mr{RW}}(\nu).\]
Therefore, $\dim\mu=\dim_{\mr{LY}}\mu\geqslant s-\ve.$

To complete the proof, it remains to show the identity between the Furstenberg entropy and the random walk entropy. 

\begin{lem}
    Let $\nu$ be a finitely supported probability measure on $\Gamma$ such that $G_\nu$ is Zariski dense. Let $\mu$ be the unique $\nu$-stationary measure on $\PP(\RR^3),$ then $h_{\mr F}(\mu,\nu)=h_{\mr{RW}}(\nu).$
\end{lem}
\begin{proof}
    As discussions after Definition \ref{def: nice word}, $X\sm X_{\mr{nice}}$ is a countable set and hence a $\mu$-null set. We consider the space $B=\SL_3(\RR)^{\times \NN}$ endowing with the probability measure $\nu^{\NN}.$ For almost every $b=(b_1,b_2,\cdots)\in B,$ we consider the Furstenberg boundary $\mu_b$ given by the weak* limit
\[\lim_{n\to\infty} (b_1b_2\cdots b_{n-1})_*\mu,\]
which is a Dirac measure \cite[Lemma 4.5]{benoist_random_2016} and denoted by $\delta_{\xi(b)}.$ Because $\mu=\int \delta_{\xi(b)}\dr\nu^{\NN}(b)$ and $\mu(X\sm X_{\mr{nice}})=0,$  there exists a conull set $\Omega\sbs B$ such that $\xi(b)\in X_{\mr{nice}}$ for every $b\in\Omega.$ 

Note that $\nu$ also induces a random walk on the flag variety $\cF=\cF(\RR^3)$ with a unique stationary measure $\nu_{\cF}.$ We can also consider the Furstenberg boundary on the flag variety. Then for almost every $b\in B,$ we can associate a Dirac measure $(\mu_\cF)_b=\delta_{\xi_\cF(b)}.$ Denote $\xi_\cF(b)=(\xi(b),\xi_2(b))\in\cF(\RR^3),$ where $\xi_2(b)$ is a two dimensional subspace in $\RR^3.$
Then for every sequence of positive numbers $(\chi_n),$ the image of every nonzero limit of $(\chi_n b_0b_1\cdots b_{n-1})$ in $\mr{End}(\wedge^2\R^3)$ is $\xi_2(b).$ 
We aim to show that $\xi_\cF(b)$ is uniquely determined by $\xi(b)$ for every $b\in \Omega .$

Let $b\ne b'\in\Omega,$ then $\xi(b)=\xi(b')\in X_{\mr{nice}}.$ Recall that every point in $X_{\mr{nice}}$ is uniquely coding by an element in $\lb{1,2,3}^\NN.$ We know that $b'=(b_1',b_2',\cdots)$ and $b=(b_1,b_2,\cdots)$ are different partitions of a same infinite word $A_{i_1}A_{i_2}\cdots A_{i_n}\cdots$. Since $\nu$ is finitely supported, for every $n\geqslant 0$ we can find $m_n\geqslant 0$ such that
    \[|b_1b_2\cdots b_n| \leq |b_
    1' b_2'\cdots b_{m_n}'|\leq |b_1b_2\cdots b_n|+N_{\nu}, \]
where $|\cdot|$ denotes the word length in the semigroup $\Gamma$ with respect to the free generating set $\lb{A_1,A_2,A_3},$ where $N_\nu$ is a constant only depending on $\nu.$

Passing to a subsequence, we can assume that $(b_1'\cdots b_{m_n}')^{-1}b_1\cdots b_{n}$ is a fixed element in $\SL_3(\RR).$ 
Taking $\chi_n=\|b_1b_2\cdots b_{n}\|^{-1}$ and passing to a subsequence if necessary, we also assume that $(\chi_n b_1b_2\cdots b_{n})$ admits the nonzero limit in $\mr{End}(\wedge^2\R^3).$ Meanwhile, $(\chi_n b_1'\cdots b_{m_n}')$ converges to this limit composing with an element in $\SL_3(\RR).$ Therefore, these two limits has the same image. By our discussions on $\xi_2,$ we obtain $\xi_2(b')=\xi_2(b)$ and hence
\begin{equation}\label{equ:equal b b'}
    \xi_\calF(b)=\xi_\calF(b').
\end{equation}


    Now we consider a conditional measure at a $\mu$-full measure set $\xi(\Omega)\sbs\PP(\RR^3)$. For every $x\in\xi(\Omega)$, choose an element  $b\in\Omega$ with $\xi(b)=x\in \P(\R^3)$ and let 
    \[\mu^x=\delta_{\xi_\calF(b)}. \]
    Due to \cref{equ:equal b b'}, we have 
    \[\mu_\calF=\int_\Omega \delta_{\xi_\calF(b)}\dd\nu^{\N}(b)=\int_{\P(\R^3)} \mu^x\dd\mu(x). \]
    So the family of measures $\mu^x$ is actually the disintegration of $\mu_\cF$ with respect to the natural projection $\pi:\cF\to\PP(\RR^3).$ Hence we obtain the trivial-fiber property. Now we can apply a same argument as in the end of the proof \cref{prop:equal.entropy.geometry} by using relative measure preserving property. We obtain the desired equality between the Furstenberg entropy and the random walk entropy.
\end{proof}

\appendix
\section{Proof of \cref{cor: 1.5}}

In the appendix, we 
prove \cref{cor: 1.5}, that is
\begin{equation}\label{equ:geq1.5}
    s_{\mathrm{A}}(\Gamma_{\sR})\geq 3/2.
\end{equation}

The secret sauce is to study the semigroup $\Gamma_0=\langle A_1,A_2\rangle$. Let $I=I(E_1,E_2)$ be the arc given by $\{\R(se_1+te_2):\ s,t\in\R_{\geq 0}\}$, which is preserved by $\Gamma_0$. 
The idea is as follows.  We can view the semigroup $\Gamma_0$ in two aspects:  $\Gamma_0$, as a subsemigroup in $\mathrm{GL}(E_1\oplus E_2)$, has critical exponent $1$ due to its limit set is the whole $I$; $\Gamma_0$ is also a subsemigroup of $\SL_3(\R)$. Then, we compare the singular values in these two settings and deduce that its affinity exponent is at least $1.5$. The argument is similar to the one of the dimension jump of the limit sets of representations in Barbot's component as in \cite{LPX}. The difficulty in the proof is due to the non-uniform hyperbolic behaviour of the Rauzy semigroup, and we borrow estimates from \cref{sec:upper rauzy} to deal with this issue.

\begin{lem}\label{lem:a1a2}
    There exists $\epsilon>0$ such that for any $\gamma\in \Gamma_0 $ with the last two digits different, we have
    \begin{align}
        \sigma_2(\gamma)\geq \epsilon,\ \epsilon |\gamma I|\leq 1/\sigma_1(\gamma)^2,
    \end{align}
    where $|\gamma I|$ is the length of the arc $\gamma I$.
\end{lem}

\begin{proof}
    It is a consequence of \cref{lem: crucial proj geo lem} and \cref{lem: proj geo lem}.

    Due to $\gamma$ preserving the subspace generate by $E_1,E_2$, and its restriction having determinant 1, we obtain from \cref{lem: crucial proj geo lem} that
    \begin{equation}\label{equ:gammaI}
     |\gamma I|=d(\gamma E_1,\gamma E_2)=\frac{\|\gamma e_1\wedge \gamma e_2\|}{\|\gamma e_1\|\|\gamma e_2\|}=\frac{1}{\|\gamma e_1\|\|\gamma e_2\|}\leq 1/(\epsilon_2\sigma_1(\gamma))^2. 
    \end{equation}

    Due to the same computation as in \cref{lem: proj geo lem} and from definition, we have 
    \[ \mathrm{Area}  (\gamma\Delta)\gg\|\gamma e_1\wedge \gamma e_2\wedge \gamma e_3 \|/\prod_{1\leq i\leq 3}\omega(\gamma e_i) \gg 1/\sigma_1(\gamma)^3. \]
    Taking into accout \cref{equ:gammaI},
    \[ \max\{ |\gamma I(E_2,E_3)|,|\gamma I(E_1,E_3)| \}\geq \mathrm{Area}(\gamma\Delta)/|\gamma I|\gg 1/\sigma_1(\gamma).  \]
    Combining with \cref{lem: proj geo lem}, we have
    \[ \sigma_2(\gamma)\gg  \sigma_1(\gamma)\diam(\gamma\Delta)\geq \sigma_1(\gamma)\max\{ |\gamma I(E_2,E_3)|,|\gamma I(E_1,E_3)| \}\gg 1.  \]
    The proof is complete.
\end{proof}

\begin{proof}[Proof of \cref{equ:geq1.5}]
For elements satisfying \cref{lem:a1a2}, we have
\[ \varphi_{3/2}(\gamma)=\left(\frac{\sigma_2}{\sigma_1}\right)(\gamma)\left(\frac{\sigma_3}{\sigma_1}\right)^{1/2}(\gamma)= \frac{\sigma_2(\gamma)^{1/2}}{\sigma_1(\gamma)^2}\geq \frac{\epsilon^{1/2}}{\sigma_1(\gamma)^2}\geq \epsilon^2|\gamma I|. \]
Therefore
\begin{align*}
    \sum_{\gamma\in\Gamma_\sR} \varphi_{3/2}(\gamma)\geqslant \sum_{\gamma\in \Gamma_0} \varphi_{3/2}(\gamma) &\geq \sum_{\gamma\in \Gamma_0 , \text{ last two digits different}} \varphi_{3/2}(\gamma)\\
    &\geq \epsilon^2 \sum_{\gamma\in \Gamma_0 , \text{ last two digits different}} |\gamma I|.
\end{align*}
Due to \cref{lem:Xnice}, every point in $I\cap X_{\mathrm{nice}}$ is covered infinity many times by the cover 
\[\{\gamma I,\ \gamma\in \Gamma_0  \text{ whose last two digits are different} \}.\]
Since $I-(I\cap X_{\mathrm{nice}})$ is countable, the righthand side of the last inequality is infinite, which implies $s_{\mathrm{A}}(\Gamma_{\sR})\geq 1.5$.    
\end{proof}


\begin{small}

\addcontentsline{toc}{section}{References}
\bibliography{bibfile}

\bigskip
\noindent 
    Yuxiang Jiao.
	{\textit{Peking University, No.5 Yiheyuan Road, Haidian District, Beijing, China.}}  \\
	email: {\tt ajorda@pku.edu.cn} 
\bigskip
    
 \noindent 
	Jialun Li. {\it CNRS-Centre de Math\'ematiques Laurent Schwartz, \'Ecole Polytechnique, Palaiseau, France.}  \\
	email: {\tt jialun.li@cnrs.fr} 
		
		\bigskip   

   \noindent 
   Wenyu Pan.
	{\textit{University of Toronto, 40 St. George St., Toronto, ON, M5S 2E4, Canada.}}  \\
	email: {\tt wenyup.pan@utoronto.ca} 
		
		\bigskip   

   \noindent
   Disheng Xu. 
	{\textit{Great Bay University,  Songshanhu International Community, Dongguan, Guangdong, 523000, China.}}  \\
	email: {\tt xudisheng@gbu.edu.cn} 
		
		\bigskip   
\end{small}
	
\end{document}